\documentclass{article}
\usepackage{latexsym,amsfonts,amsmath,amsthm,makeidx}
\usepackage{makeidx}
\makeindex
\newtheorem{definition}{Definition}[section]
\newtheorem{theorem}{Theorem}[section]
\newtheorem{lemma}{Lemma}[section]
\newtheorem{corollary}{Corollary}[section]

\newtheorem{remark}{Remark}[section]
\newcommand{\RN}{\mathbb R^N}
\newcommand{\Om}{\Omega}

\newcommand{\iy}{\infty}

\newcommand{\dd}{\delta}

\newcommand{\la}{\lambda}
\newcommand{\La}{\Lambda}

\newcommand{\R}{\mathbb R}
\newcommand{\al}{\alpha}

\newcommand{\wH}{\widetilde H}

\newcommand{\bb}{\beta}

\newcommand{\e}{\varepsilon}
\newcommand{\vp}{\varphi}

\newcommand{\bt}{\begin{theorem}}
\newcommand{\et}{\end{theorem}}
\newcommand{\bl}{\begin{lemma}}
\newcommand{\el}{\end{lemma}}
\newcommand{\bd}{\begin{definition}}
\newcommand{\ed}{\end{definition}}
\newcommand{\bc}{\begin{corollary}}
\newcommand{\ec}{\end{corollary}}
\newcommand{\bp}{\begin{proof}}
\newcommand{\ep}{\end{proof}}
\newcommand{\bx}{\begin{example}}
\newcommand{\ex}{\end{example}}
\newcommand{\bi}{\begin{exercise}}
\newcommand{\ei}{\end{exercise}}
\newcommand{\bo}{\begin{prop}}
\newcommand{\eo}{\end{prop}}
\newcommand{\br}{\begin{remark}}
\newcommand{\er}{\end{remark}}
\newcommand{\be}{\begin{equation}}
\newcommand{\ee}{\end{equation}}
\newcommand{\ba}{\begin{align}}
\newcommand{\ea}{\end{align}}

\newcommand{\ga}{\gamma}
\newcommand{\sg}{\sigma}
\newcommand{\bean}{\begin{eqnarray*}}
\newcommand{\eean}{\end{eqnarray*}}

\numberwithin{equation}{section}

\begin{document}

\title{\bf Infinitely many sign-changing and semi-nodal solutions for a nonlinear
Schr\"{o}dinger system\thanks{Zou is supported by NSFC (11025106, 11371212, 11271386) and the Both-Side Tsinghua Fund.  E-mail:
chenzhijie1987@sina.com (Chen);\quad cslin@math.ntu.edu.tw (Lin, corresponding author); \quad wzou@math.tsinghua.edu.cn (Zou)}}
\date{}
\author{{\bf Zhijie Chen$^{1}$,   Chang-Shou Lin$^{2}$,
  Wenming Zou$^{3}$}\\
\footnotesize {\it $^{1, 3}$Department of Mathematical Sciences, Tsinghua University, Beijing 100084, China}\\
\footnotesize{\it $^{1}$Current address: Center for Advanced Study in Theoretical Science,}\\
\footnotesize{\it National Taiwan
University, Taipei 106, Taiwan}\\
\footnotesize {\it $^2$Taida Institute for Mathematical Sciences, Center for Advanced Study in}\\
\footnotesize {\it Theoretical Science, National Taiwan
University, Taipei 106, Taiwan} }

\maketitle

\begin{center}
\begin{minipage}{120mm}
\begin{center}{\bf Abstract}\end{center}

We study the following coupled
Schr\"{o}dinger equations which have appeared as several models from mathematical physics:
\begin{displaymath}
\begin{cases}-\Delta u_1 +\la_1 u_1 =
\mu_1 u_1^3+\beta u_1 u_2^2, \quad x\in \Omega,\\
-\Delta u_2 +\la_2 u_2 =\mu_2 u_2^3+\beta u_1^2 u_2,  \quad   x\in
\Om,\\
u_1=u_2=0  \,\,\,\hbox{on \,$\partial\Om$}.\end{cases}\end{displaymath}
Here $\Om$ is a smooth bounded domain in $\R^N (N=2, 3)$ or $\Om=\RN$,
$\la_1,\, \la_2$, $\mu_1,\,\mu_2$ are all positive constants and the coupling constant $\bb<0$. We show that
this system has infinitely many sign-changing solutions. We also obtain infinitely
many semi-nodal solutions in the following sense: one component changes
sign and the other one is positive.
The crucial idea of our proof, which has never been used for this system
before, is to study a new problem with two constraints.
Finally, when $\Om$ is a bounded domain, we show that this system has
a least energy sign-changing solution,
both two components of which have exactly two nodal domains, and
we also study the asymptotic behavior of solutions as $\beta\to -\infty$ and
phase separation is expected.\\


\end{minipage}
\end{center}

\section{Introduction}

In this paper we study solitary wave solutions of the coupled Gross-Pitaevskii equations (cf. \cite{CLLL}):
\be\label{eq1}
\begin{cases}-i \frac{\partial}{\partial t}\Phi_1=\Delta \Phi_1+
\mu_1 |\Phi_1|^2 \Phi_1+\beta |\Phi_2|^2\Phi_1, \quad x\in \Omega, \,\,t>0,\\
-i \frac{\partial}{\partial t}\Phi_2=\Delta \Phi_2+
\mu_2|\Phi_2|^2 \Phi_2+\beta |\Phi_1|^2\Phi_2,  \quad x\in \Omega, \,\,t>0,\\
\Phi_j=\Phi_j(x,t)\in\mathbb{C},\quad j=1,2,\\
\Phi_j(x,t)=0,  \quad x\in \partial\Om, \,\,t>0, \,\,j=1,2,\end{cases}\ee
where $\Om=\RN (N=2, 3)$ or $\Om\subset\RN$ is a smooth bounded domain, $i$ is the imaginary unit,
$\mu_1,\mu_2 >0$ and $\beta\neq 0$ is a coupling constant.
System (\ref{eq1}) arises
in mathematical models from several physical
phenomena, especially in nonlinear optics.
Physically, the solution $\Phi_j$ denotes the $j^{th}$ component of the beam
in Kerr-like photorefractive media (cf. \cite{AA}).
The positive constant $\mu_j$ is for self-focusing in the $j^{th}$ component of
the beam, and the coupling constant $\beta$ is
the interaction between the two components of the beam. Problem (\ref{eq1}) also arises
in the Hartree-Fock theory for a double
condensate, i.e., a binary mixture of Bose-Einstein
condensates in two different hyperfine states $|1\rangle$ and $|2\rangle$ (cf. \cite{EGBB}).
Physically, $\Phi_j$ are the corresponding condensate
amplitudes, $\mu_j$ and $\beta$ are the intraspecies and
interspecies scattering lengths. Precisely, the sign of $\mu_j$ represents
the self-interactions of the single state $|j\rangle$.
If $\mu_j>0$ as considered here, it is called the focusing case, in
opposition to the defocusing case where $\mu_j<0$. Besides,
the sign of $\beta$ determines whether the interactions of
states $|1\rangle$ and $|2\rangle$
are repulsive or attractive, i.e., the interaction is
attractive if $\beta>0$, and the interaction is repulsive
if $\beta<0$, where the two states are in strong competition when $\beta$ is negative and very large.

To obtain solitary wave solutions of system (\ref{eq1}),
we set $\Phi_j(x, t)=e^{i\la_j t}u_j(x)$ for $j=1, 2$, where $u_j(x)$ are real-valued functions. Then system (\ref{eq1})
is reduced to the following elliptic system
\be\label{eq2}
\begin{cases}-\Delta u_1 +\la_1 u_1 =
\mu_1 u_1^3+\beta u_1 u_2^2, \quad x\in \Omega,\\
-\Delta u_2 +\la_2 u_2 =\mu_2 u_2^3+\beta u_1^2 u_2,  \quad   x\in\Om,\\
u_1=u_2=0  \,\,\,\hbox{on \,$\partial\Om$}.\end{cases}\ee
Here, for the case $\Om=\RN$, the boundary condition $u_1=u_2=0$ on $\partial\Om$ means
$$u_1(x), u_2(x)\to 0\quad\text{as}\,\,\,|x|\to +\iy.$$
It is well known that finite energy solutions
of (\ref{eq2}) correspond to the critical points of $C^2$ functional $E_\bb: H_0^1(\Om)\times H_0^1(\Om)\to \R$
given by
{\allowdisplaybreaks
\begin{align}\label{eq2-1}
E_\bb(u_1, u_2):=&\frac{1}{2} \int_{\Om}(|\nabla u_1|^2+\la_1 u_1^2+|\nabla u_2|^2+\la_2 u_2^2)dx\nonumber\\
&-\frac{1}{4}\int_{\Om}(\mu_1 u_1^4+\mu_2u_2^4)dx-\frac{\beta}{2} \int_{\Om}u_1^2u_2^2dx.
\end{align}
}%

\begin{definition}\label{definition}
We call a solution $(u_1, u_2)$ {\it nontrivial}
if $u_j\not\equiv 0$ for $j=1, 2$, a solution $(u_1, u_2)$ {\it semi-trivial} if $(u_1, u_2)$ is type
of $(u_1, 0)$ or $(0, u_2)$. A solution $(u_1, u_2)$ is
called {\it positive} if $u_j > 0$ in $\Om$ for $j=1, 2$, a solution $(u_1, u_2)$
{\it sign-changing} if both $u_1$ and $u_2$ change sign,
a solution $(u_1, u_2)$ {\it semi-nodal} if one component is positive and the other one changes sign.
\end{definition}

\begin{definition}
A nontrivial solution $(u_1, u_2)$ is called a least energy solution,
if it has the least energy among all nontrivial solutions, i.e., $E_\bb(u_1, u_2)\le E_\bb(v_1, v_2)$ for any nontrivial solution $(v_1, v_2)$ of (\ref{eq2}).
A sign-changing solution $(u_1, u_2)$ is called a least energy sign-changing solution,
if it has the least energy among all sign-changing solutions.
\end{definition}

In the last decades, system (\ref{eq2}) has received
great interest from many mathematicians.
When $\Om$ is the entire space $\RN$, the existence of
least energy and other finite energy solutions of (\ref{eq2})
was studied in \cite{AC2, BW, BWW, CZ1, LW1, LW, LWZ, MMP, MMP1, S, TV} and references therein.
In particular, when $\bb>0$ is sufficiently large, multiple radially symmetric sign-changing solutions
of (\ref{eq2}) were obtained in \cite{MMP1}, where the radial symmetry of $\RN$ plays a crucial role in their proof. Under
assumptions $\la_i>0, \mu_i>0$ and $\bb<0$, Liu and Wang \cite{LWZ} proved that
system (\ref{eq2}) has infinitely many nontrivial solutions. In fact, they studied a general $m$-coupled system ($m\ge 2$). Remark that
whether solutions obtained in \cite{LWZ} are positive or sign-changing
are not known.

When $\Om\subset \RN (N=2, 3)$ is a smooth bounded domain, there are also many papers studying (\ref{eq2}).
Lin and Wei \cite{LW2} proved that a least energy solution
of (\ref{eq2}) exists within the range $\bb\in (-\iy, \bb_0)$, where $0<\bb_0<\sqrt{\mu_1\mu_2}$.
In case where $\la_1=\la_2>0$, $\mu_1=\mu_2>0$ and $\bb\le-\mu_1$, Dancer, Wei and Weth \cite{DWW} proved
the existence of infinitely many positive solutions of (\ref{eq2}),
while the same result was proved for the case $\la_1=\la_2<0$ by Noris and Ramos \cite{NR}.
When $\Om$ is a ball, an interesting multiplicity result on
positive radially symmetric solutions was given in \cite{WW2}.
Remark that, since $\la_1=\la_2$ and $\mu_1=\mu_2$, so system (\ref{eq2})
is invariant under the transformation $(u_1, u_2)\mapsto (u_2, u_1)$, which
plays a crucial role in \cite{DWW, NR, WW2}. Later, by using a global bifurcation approach,
the result of \cite{WW2} was
reproved by \cite{BDW} without requiring the symmetric condition $\mu_1=\mu_2$, but in their proof
the assumption $\la_1=\la_2$ plays a crucial role. Under
assumptions $\la_i>0, \mu_i>0$ and $\bb<0$ without
requiring $\la_1=\la_2$ or $\mu_1=\mu_2$, Sato and Wang \cite{SW} proved that
system (\ref{eq2}) has infinitely many semi-positive solutions (i.e., at least one component is positive).
Note that all the papers mentioned above deal with the subcritical case $N\le 3$ (i.e., the cubic nonlinearities are all of subcritical growth). Recently,
Chen and Zou \cite{CZ} studied the existence and properties of
least energy solutions of (\ref{eq2}) in the critical case $N=4$.

In a word, for $N=2, 3$, {\it a natural question, which seems to be still
open for both the entire space case and the bounded domain case,
is whether (\ref{eq2}) has infinitely many sign-changing solutions when $\bb<0$}.
This is expected by many experts but no proof has yet been obtained.
Here we can give a positive answer to this open question. Since the results in the entire space case are slightly different
from those in the bounded domain case, in this section
we only state our results in the bounded domain case for the sake of brevity.
The results in the entire space case will be given in Section 6. Our first result is as follows.

\bt\label{th1}Let $N=2, 3$, $\Om\subset \RN$ is a smooth bounded
domain, $\la_1, \la_2$, $\mu_1, \mu_2>0$ and $\bb<0$. Then (\ref{eq2}) has infinitely
many sign-changing solutions $(u_{n, 1}, u_{n, 2})$ such that
$$\|u_{n, 1}\|_{L^\iy(\Om)}+\|u_{n,2}\|_{L^\iy(\Om)}\to+\iy\quad\hbox{as}\,\,\,n\to+\iy.$$\et

\br\label{rmk}Comparing with \cite{BDW, DWW, NR, WW2} where infinitely many positive
solutions were obtained,  we do not need any symmetric assumptions $\la_1=\la_2$
or $\mu_1=\mu_2$.\er

\br\label{rmk2} All the papers mentioned above and this paper deal with the
focusing case $\mu_1, \mu_2>0$. For the defocusing case $\mu_1, \mu_2<0$, related
results can be seen in \cite{CLLL, NTTV1, NTTV, TT, TT1}. In particular, Tavares and Terracini \cite{TT}
studied the following general $m$-coupled system
\be\label{eq3}
\begin{cases}-\Delta u_j -\mu_j u_j^3-\bb u_j\sum_{i\neq j} u_i^2 =\la_{j, \bb} u_j,\\
u_j\in H_0^1(\Om),\quad j=1,\cdots, m,\end{cases}\ee
where $\Om$ is a smooth bounded domain, $\bb<0$ and $\mu_j\le 0$. Then \cite[Theorem 1.1]{TT} says that for each
fixed $\bb<0$ and $\mu_1,\cdots,\mu_m\le 0$, there exist
infinitely many $\la=(\la_{1, \bb},\cdots,\la_{m, \bb})\in \R^m$ and $u=(u_1, \cdots, u_m)\in H_0^1(\Om, \R^m)$
such that $(u, \la)$ are sign-changing solutions of (\ref{eq3}).
That is, for each fixed $\bb<0$ and $\mu_1,\cdots,\mu_m\le 0$, $\la_{j, \bb}$ is not
fixed a {\it priori} and appears as a Lagrange multiplier in \cite{TT}. Our result is different from
\cite[Theorem 1.1]{TT} on two aspects: one is that we deal with the focusing case $\mu_j>0$, the other one
is that $\la_j, \mu_j$ and $\bb$ are all fixed constants in Theorem \ref{th1}.
To the best of our knowledge, our result for system (\ref{eq2}) is new.
\er

As pointed out before,
Lin and Wei \cite{LW2} proved for $\bb\in (-\iy, \bb_0)$ that (\ref{eq2}) has a least energy solution
which turns out to be a positive solution. Since (\ref{eq2})
has infinitely many sign-changing solutions for any $\bb<0$, another natural question is whether (\ref{eq2}) has a least energy sign-changing
solution, which has not been studied before. Here we can prove the
following result.

\bt\label{th2} Let assumptions in Theorem \ref{th1} hold. Then (\ref{eq2}) has
a least energy sign-changing solution $(u_1, u_2)$. Moreover, both $u_1$ and $u_2$
have exactly two nodal domains.\et

Theorems \ref{th1} and \ref{th2} are both concerned with sign-changing solutions. Besides positive solutions (see \cite{BDW, DWW, WW2}) and
sign-changing solutions, as defined in Definition \ref{definition}, it is natural to suspect that (\ref{eq2}) may have
semi-nodal solutions. Here we
can prove the following result.

\bt\label{th3} Let assumptions in Theorem \ref{th1} hold.
Then (\ref{eq2}) has infinitely many semi-nodal solutions $\{(u_{n,1}, u_{n, 2})\}_{n\ge 2}$ such that
\begin{itemize}
\item[$(1)$] $u_{n, 1}$ changes sign and $u_{n, 2}$ is positive;
\item[$(2)$] $\|u_{n, 1}\|_{L^\iy(\Om)}+\|u_{n,2}\|_{L^\iy(\Om)}\to+\iy\,\,\hbox{as}\,\,n\to+\iy$;
\item[$(3)$] $u_{n, 1}$ has at most $n$ nodal domains. In particular, $u_{2, 1}$ has exactly
two nodal domains, and $(u_{2,1}, u_{2,2})$ has the least energy among all nontrivial
solutions whose first component changes sign.
\end{itemize}
\et

\br\label{rmk3-1} Recently, we found that \cite[Theorem 0.1]{SW} proved that (\ref{eq2}) has infinitely many semi-nodal
solutions for any $\bb\in (-\sqrt{\mu_1\mu_2}, 0)$. Theorem \ref{th3} improves \cite[Theorem 0.1]{SW} on two aspects: one is that we can
obtain infinitely many semi-nodal solutions for $\bb\le-\sqrt{\mu_1\mu_2}$;
the other one is that, in \cite{SW} no properties of the form $(3)$ can be
obtained by their approach. Our proofs in this paper
are completely different from \cite{SW}.\er

\br\label{rmk3} Similarly, we can prove that (\ref{eq2}) has infinitely many semi-nodal
solutions $\{(v_{n,1}, v_{n, 2})\}_{n\ge 2}$
such that $v_{n, 1}$ is positive, $v_{n, 2}$ changes sign and has at most $n$ nodal domains.
In the symmetric case where $\la_1=\la_2$ and $\mu_1=\mu_2$, $(u_{n,1}, u_{n,2})$ obtained
 in Theorem \ref{th3} and $(v_{n, 1}, v_{n, 2})$ may be the same solution in the sense of $u_{n,1}=v_{n,2}$ and $u_{n,2}=v_{n,1}$.
 However, if either $\la_1\neq\la_2$ or
$\mu_1\neq \mu_2$, then $(u_{n,1}, u_{n,2})$
and $(v_{n,1}, v_{n,2})$ are really different solutions.\er

We give some notations here. Throughout this paper,
we denote the norm of $L^p(\Om)$ by $|u|_p =
(\int_{\Om}|u|^p\,dx)^{\frac{1}{p}}$, the norm of $H^1_0(\Om)$ by $\|u\|^2=\int_{\Om}(|\nabla u|^2+u^2)\,dx$ and positive constants
(possibly different in different places) by $C, C_0, C_1, \cdots$. Denote
$\|u\|_{\la_i}^2:=\int_{\Om} (|\nabla u|^2 + \la_i u^2)\,dx$
for convenience. Since we assume $\la_1, \la_2>0$ here, $\|\cdot\|_{\la_i}$ are
equivalent norms to $\|\cdot\|$. Define $H:=H_0^1(\Om)\times H_0^1(\Om)$ with norm $\|(u_1, u_2)\|_H^2:=\|u_1\|_{\la_1}^2+\|u_2\|_{\la_2}^2$.

The rest of this paper is organized as follows. In Section 2 we give the proof of Theorem \ref{th1}.
The main idea of this proof is inspired by \cite{TT},
where a new notion of {\it vector genus} introduced by \cite{TT} will be used to define appropriate minimax values.
Some arguments in our proof are borrowed from \cite{TT} with modifications.
Remark that the ideas in \cite{TT} can not be used directly, and here we will give
some new ideas. For example, to obtain nontrivial solutions of (\ref{eq2}), the crucial idea in this paper is
turning to study a new problem with two constraints.
Somewhat surprisingly, up to our knowledge, this natural idea has never been
used for (\ref{eq2}) in the literature; see Remark \ref{remark} below. In Section 3
we will use general Nehari type manifolds to prove Theorem \ref{th2}.
By giving some modifications to arguments in Sections 2 and 3, we will prove
Theorem \ref{th3} in Section 4.

In Section 5, we will study the limit behavior of solutions obtained here
as $\bb\to-\iy$ by applying results in \cite{NTTV} directly.
It turns out that components of the limiting profile tend to separate in different regions.
This phenomena, called {\it phase separation}, has been well studied for $L^{\iy}$-bounded positive
solutions of (\ref{eq2}) in the case $N=2, 3$
by \cite{WW2, WW3, NTTV}. For other kinds of elliptic systems with strong competition,
phase separation has also been well studied, we refer to \cite{CL0, CTV} and references therein.
The main result of Section 5 is Theorem \ref{th4}.

Finally in Section 6, we will introduce existence results of infinitely many radially
symmetric sign-changing and semi-nodal solutions in the entire
space case. The main results are Theorems \ref{th6-1} and \ref{th6-3}.
The main ideas of the proof are the same as those in Sections 2-4.
However, we will see that some ideas and arguments are quite different from those in the bounded domain case.

After this paper was submitted, we learned from Z.-Q. Wang of a recent work \cite{LLW}, where infinitely many sign-changing solutions of the general $m$-coupled system ($m\ge 2$) were obtained via a quite different method. We should point out that our approach here also works for the general $m$-coupled system ($m\ge 2$).

\section{Proof of Theorem \ref{th1}}
\renewcommand{\theequation}{2.\arabic{equation}}

In the sequel we assume that assumptions in Theorem \ref{th1} hold.
Since we are only concerned with nontrivial solutions,
we denote $\wH:=\{(u_1, u_2)\in H : u_i\neq 0\,\,\hbox{for}\,\, i=1, 2\}$, which is an open subset of $H$. Write $\vec{u}=(u_1, u_2)$ for convenience.

\bl\label{lemma1}For any $(u_1, u_2)\in\wH$, if
\be\label{eq2-2}|\bb|^2\left(\int_{\Om}u_1^2 u_2^2\right)^2\ge \mu_1\mu_2|u_1|_4^4 |u_2|_4^4,\ee
then
$$\sup_{t_1, t_2\ge 0}E_\bb (\sqrt{t_1}u_1, \,\sqrt{t_2} u_2)=+\iy.$$\el

\begin{proof} By (\ref{eq2-2}) there exists $\al>0$ such that
$$\al |\bb|\int_{\Om}u_1^2 u_2^2\ge \mu_1|u_1|_4^4,\quad \frac{1}{\al} |\bb|\int_{\Om}u_1^2 u_2^2\ge \mu_2|u_2|_4^4,$$
which implies
{\allowdisplaybreaks
\begin{align*}
E_\bb(\sqrt{t}u_1, \sqrt{\al t}u_2)=&\frac{1}{2} t\|u_1\|_{\la_1}^2+
\frac{1}{2}\al t\|u_2\|_{\la_2}^2\\
&-\frac{1}{4}\left(t^2\mu_1 |u_1|_4^4+\al^2 t^2\mu_2|u_2|_4^4\right)+\frac{1}{2}\al t^2|\beta| \int_{\Om}u_1^2u_2^2\\
\ge & \frac{1}{2} t\|u_1\|_{\la_1}^2+
\frac{1}{2}\al t\|u_2\|_{\la_2}^2\to+\iy\quad\hbox{as $t\to+\iy$.}
\end{align*}
}%
This completes the proof.\end{proof}

\bl\label{lemma2}For any $\vec{u}=(u_1, u_2)\in\wH$, if
\be\label{eq2-3}|\bb|^2\left(\int_{\Om}u_1^2 u_2^2\right)^2< \mu_1\mu_2|u_1|_4^4 |u_2|_4^4,\ee
then system
\be\label{eq2-4}
\begin{cases}\|u_1\|_{\la_1}^2 =
t_1\mu_1 |u_1|_4^4-t_2|\beta|\int_{\Om} u_1^2 u_2^2\\
\|u_2\|_{\la_2}^2 =
t_2\mu_2 |u_2|_4^4-t_1|\beta|\int_{\Om} u_1^2 u_2^2\end{cases}\ee
has a unique solution
\be\label{eq2-5}
\begin{cases}t_1(\vec{u})=\frac{\mu_2|u_2|_4^4\|u_1\|_{\la_1}^2+|\bb|\|u_2\|_{\la_2}^2\int_{\Om}u_1^2u_2^2}{\mu_1\mu_2|u_1|_4^4 |u_2|_4^4-|\bb|^2(\int_{\Om}u_1^2 u_2^2)^2}>0\\
t_2(\vec{u})=\frac{\mu_1|u_1|_4^4\|u_2\|_{\la_2}^2+|\bb|\|u_1\|_{\la_1}^2\int_{\Om}u_1^2u_2^2}{\mu_1\mu_2|u_1|_4^4 |u_2|_4^4-|\bb|^2(\int_{\Om}u_1^2 u_2^2)^2}>0.\end{cases}\ee
Moreover,
{\allowdisplaybreaks
\begin{align}\label{eq2-6}\sup_{t_1, t_2\ge 0}&E_\bb \left(\sqrt{t_1}u_1, \,\sqrt{t_2} u_2\right)=E_\bb\left(\sqrt{t_1(\vec{u})}u_1, \sqrt{t_2(\vec{u})}u_2\right)\nonumber\\
&=\frac{1}{4}\left(t_1(\vec{u})\|u_1\|_{\la_1}^2+t_2(\vec{u})\|u_2\|_{\la_2}^2\right)\nonumber\\
&=\frac{1}{4}\frac{\mu_2|u_2|_4^4\|u_1\|_{\la_1}^4
+2|\bb|\|u_1\|_{\la_1}^2\|u_2\|_{\la_2}^2\int_{\Om}u_1^2u_2^2+\mu_1|u_1|_4^4\|u_2\|_{\la_2}^4}{\mu_1\mu_2|u_1|_4^4 |u_2|_4^4-|\bb|^2(\int_{\Om}u_1^2 u_2^2)^2}
\end{align}
}%
and $(t_1(\vec{u}), t_2(\vec{u}))$ is the unique maximum point of $E_\bb (\sqrt{t_1}u_1, \sqrt{t_2}u_2)$.
\el

\begin{proof} It suffices to prove (\ref{eq2-6}). Recall that $(t_1(\vec{u}), t_2(\vec{u}))$
is the solution of (\ref{eq2-4}), we deduce that
{\allowdisplaybreaks
\begin{align*}
2t_1 t_2|\bb|\int_{\Om}u_1^2u_2^2 &\le t_1^2\frac{t_2(\vec{u})}{t_1(\vec{u})}|\bb|\int_{\Om}u_1^2 u_2^2
+t_2^2\frac{t_1(\vec{u})}{t_2(\vec{u})}|\bb|\int_{\Om}u_1^2 u_2^2\\
&=t_1^2\mu_1|u_1|_4^4+t_2^2\mu_2|u_2|_4^4-\frac{t_1^2}{t_1(\vec{u})}
\|u_1\|_{\la_1}^2-\frac{t_2^2}{t_2(\vec{u})}\|u_2\|_{\la_2}^2.
\end{align*}
}%
Hence for any $t_1, t_2\ge 0$,
{\allowdisplaybreaks
\begin{align*}
E_\bb\left(\sqrt{t_1}u_1, \sqrt{t_2}u_2\right)=&\frac{1}{2} t_1\|u_1\|_{\la_1}^2+
\frac{1}{2}t_2\|u_2\|_{\la_2}^2\\
&-\frac{1}{4}\left(t_1^2\mu_1 |u_1|_4^4+ t_2^2\mu_2|u_2|_4^4\right)+\frac{1}{2}t_1 t_2|\beta| \int_{\Om}u_1^2u_2^2\\
\le & \left(\frac{t_1}{2}-\frac{t_1^2}{4t_1(\vec{u})}\right)\|u_1\|_{\la_1}^2+
\left(\frac{t_2}{2}-\frac{t_2^2}{4t_2(\vec{u})}\right)\|u_2\|_{\la_2}^2\\
\le &\frac{1}{4}\left(t_1(\vec{u})\|u_1\|_{\la_1}^2+t_2(\vec{u})\|u_2\|_{\la_2}^2\right)\\
=& E_\bb\left(\sqrt{t_1(\vec{u})}u_1, \sqrt{t_2(\vec{u})}u_2\right).
\end{align*}
}%
Therefore (\ref{eq2-6}) holds and $(t_1(\vec{u}), t_2(\vec{u}))$ is also the unique maximum point of $E_\bb (\sqrt{t_1}u_1, \sqrt{t_2}u_2)$ in $[0, +\iy)^2$.\end{proof}

Define
{\allowdisplaybreaks
\begin{align}\label{eq2-8}
&\mathcal{M}^\ast:=\left\{\vec{u}\in H\,\,:\,\, |u_1|_4>1/2,\,\, |u_2|_4>1/2\right\};\nonumber\\
&\mathcal{M}_\bb^\ast:=\left\{\vec{u}\in \mathcal{M}^\ast\,\,:\,\, \hbox{$\vec{u}$ satisfies (\ref{eq2-3})}\right\};\nonumber\\
&\mathcal{M}_\bb^{\ast\ast}:=\left\{\vec{u}\in \mathcal{M}^\ast\,\,:\,\, \mu_1\mu_2-|\bb|^2\left(\int_{\Om}u_1^2 u_2^2\right)^2>0\right\};\nonumber\\
&\mathcal{M}:=\left\{\vec{u}\in H\,\,:\,\, |u_1|_4=1,\,\, |u_2|_4=1\right\},\quad \mathcal{M}_\bb:=\mathcal{M}\cap \mathcal{M}_\bb^\ast.
\end{align}
}%
Then $\mathcal{M}_\bb=\mathcal{M}\cap \mathcal{M}_\bb^{\ast\ast}$. By taking $\varphi_i\in C_0^{\iy}(\Om)$ such that $|\varphi_i|_4=1$ for $i=1, 2$ and $\text{supp}(\varphi_1)\cap\text{supp}(\varphi_2)=\emptyset$, we have $(\varphi_1, \varphi_2)\in \mathcal{M}_\bb$, namely $\mathcal{M}_\bb\neq\emptyset$.
It is easy to check that $\mathcal{M}^\ast$, $\mathcal{M}_\bb^*$, $\mathcal{M}_\bb^{\ast\ast}$ are all open subsets of $H$ and $\mathcal{M}$ is closed.
Define a new functional $J_\bb : \mathcal{M}^\ast\to (0, +\iy]$ by
$$J_\bb (\vec{u}):=\begin{cases}\frac{1}{4}\frac{\mu_2\|u_1\|_{\la_1}^4
+2|\bb|\|u_1\|_{\la_1}^2\|u_2\|_{\la_2}^2\int_{\Om}u_1^2u_2^2+\mu_1\|u_2\|_{\la_2}^4}{\mu_1\mu_2-|\bb|^2(\int_{\Om}u_1^2 u_2^2)^2}
&\hbox{if $\vec{u}\in \mathcal{M}^{\ast\ast}_\bb$},\\
+\iy &\hbox{if $\vec{u}\in \mathcal{M}^\ast\setminus\mathcal{M}^{\ast\ast}_\bb$}.\end{cases}$$
By the Sobolev inequality
\be\label{eq2-10}\|u\|_{\la_i}^2\ge C |u|_4^2,\quad\forall\, u\in H_0^1(\Om),\,\,\,i=1, 2,\ee
where $C$ is a positive constant, it is easy to check that $J_\bb$ is continuous on $\mathcal{M}^\ast$ and $\inf_{\mathcal{M}^\ast}J_\bb\ge C_1>0$ for
some constant $C_1$ independent of $\bb<0$.
Moreover, $J_\bb\in C^1(\mathcal{M}^{\ast\ast}_\bb, \,(0, +\iy))$,
and since any $\vec{u}\in \mathcal{M}_\bb$ is an interior point of $\mathcal{M}_\bb^{\ast\ast}$, a direct computation and (\ref{eq2-5}) yield that
{\allowdisplaybreaks
\begin{align}
\label{eq2-11}&J_\bb'(\vec{u})(\vp, 0)=t_1(\vec{u})\int_{\Om}(\nabla u_1\nabla\vp+\la_1 u_1\vp)+t_1(\vec{u})t_2(\vec{u})|\bb|\int_{\Om}u_1 u_2^2\vp,\\
\label{eq2-11-1}&J_\bb'(\vec{u})(0,\psi)=t_2(\vec{u})\int_{\Om}(\nabla u_2\nabla\psi+\la_2 u_2\psi)+t_1(\vec{u})t_2(\vec{u})|\bb|\int_{\Om}u_1^2 u_2\psi
\end{align}
}%
hold for any $\vec{u}\in \mathcal{M}_\bb$ and $\vp,\, \psi\in H_0^1(\Om)$
(remark that (\ref{eq2-11})-(\ref{eq2-11-1}) do not hold for $\vec{u}\in\mathcal{M}_\bb^{\ast\ast}\setminus\mathcal{M}_\bb$). Note that
Lemmas \ref{lemma1} and \ref{lemma2} yield
\be\label{eq2-9}J_\bb (u_1, u_2)=\sup_{t_1, t_2\ge 0}E_\bb \left(\sqrt{t_1}u_1, \,\sqrt{t_2} u_2\right),\quad \forall\,(u_1, u_2)\in \mathcal{M}.\ee
 To obtain nontrivial solutions of (\ref{eq2}),
we turn to study the new functional $J_\bb$ restricted to $\mathcal{M}_\bb$, which is a problem with two constraints.

\br\label{remark} To obtain nontrivial solutions of (\ref{eq2}), in many papers (see \cite{CZ, DWW, LW1, LW2, S, WW2} for example), people usually turn to study nontrivial critical
points of $E_\bb$ under the following Nehari manifold type constraint
$$\big\{(u_1, u_2)\in \widetilde{H}\,:\, E_\bb'(u_1, u_2)(u_1, 0)=E_\bb'(u_1, u_2)(0, u_2)=0\big\},$$
which is actually a natural constraint for any $\bb<\sqrt{\mu_1\mu_2}$ (see \cite[Proposition 1.1]{S} for example).
To the best of our knowledge, our natural idea (i.e., to obtain nontrivial solutions of (\ref{eq2}) by studying $J_\bb|_{\mathcal{M}_\bb}$),
has never been introduced for (\ref{eq2}) in the literature.
\er

In the following, we always let $(i, j)=(1, 2)$ or $(i, j)=(2, 1)$. Recall that $t_i(\vec{u})$ is well defined for $\vec{u}\in\mathcal{M}_\bb^\ast$.
 For any $\vec{u}=(u_1, u_2)\in \mathcal{M}_\bb^*$, let $\tilde{w}_i\in H_0^1(\Om)$ be the unique
solution of the following linear problem
\be\label{eq2-012}-\Delta \tilde{w}_i+\la_i \tilde{w}_i+|\bb|t_j(\vec{u})u_j^2 \tilde{w}_i=\mu_i t_i(\vec{u})u_i^3,\quad \tilde{w}_i\in H_0^1(\Om).\ee
Since $|u_i|_4 > 1/2$, so $\tilde{w}_i\neq 0$ and
$$\int_{\Om}u_i^3 \tilde{w}_i=\frac{1}{\mu_i t_i(\vec{u})}\left(\|\tilde{w}_i\|_{\la_i}^2+|\bb|t_j(\vec{u})\int_{\Om}u_j^2 \tilde{w}_i^2\right)>0.$$
Define
\be\label{eq2-12}w_i=\al_i \tilde{w}_i,\quad\hbox{where}\,\,\al_i=\frac{1}{\int_{\Om}u_i^3 \tilde{w}_i}>0.\ee
Then $w_i$ is the unique solution of the following problem
\be\label{eq2-13}\begin{cases}
-\Delta w_i+\la_i w_i+|\bb|t_j(\vec{u})u_j^2 w_i=\al_i\mu_i t_i(\vec{u})u_i^3,\quad w_i\in H_0^1(\Om),\\
\int_{\Om}u_i^3 w_i\,dx=1.
\end{cases}\ee
Now we define an operator $K=(K_1, K_2) : \mathcal{M}_\bb^\ast\to H$ by
\be\label{eq2-14}K(\vec{u})=(K_1(\vec{u}), K_2(\vec{u})):=\vec{w}=(w_1, w_2).\ee
Define the transformations
\be\label{involution}\sg_i: H\to H\quad\hbox{by}\quad \sg_1(u_1, u_2):=(-u_1, u_2),\,\,\,\sg_2(u_1, u_2):=(u_1, -u_2).\ee
Then it is easy to check that
\be\label{eq2-15}K(\sg_i(\vec{u}))=\sg_i(K(\vec{u})),\quad i=1, 2.\ee

\bl\label{lemma3} $K\in C^1(\mathcal{M}_\bb^\ast, H)$.\el

\begin{proof} It suffices to apply the Implicit Theorem to the $C^1$ map
{\allowdisplaybreaks
\begin{align*}
&\Psi : \mathcal{M}_\bb^\ast\times H_0^1(\Om)\times \R\to H_0^1(\Om)\times\R, \quad\hbox{where}\\
&\Psi(\vec{u}, v, \al)=\left(v+(-\Delta+\la_i)^{-1}\left(|\bb|t_j(\vec{u})u_j^2 v-\al \mu_i t_i(\vec{u})u_i^3\right),\,\,\int_{\Om}u_i^3v -1\right).
\end{align*}
}%
Note that (\ref{eq2-13}) holds if and only if $\Psi(\vec{u}, w_i, \al_i)=(0, 0)$. By computing the derivative of $\Psi$ with respect to $(v, \al)$ at the
point $(\vec{u}, w_i, \al_i)$ in the direction $(\bar{w}, \bar{\al})$, we obtain a map $\Phi: H_0^1(\Om)\times \R\to H_0^1(\Om)\times\R$ given by
{\allowdisplaybreaks
\begin{align*}
\Phi(\bar{w}, \bar{\al}):=& D_{v, \al}\Psi(\vec{u}, w_i, \al_i)(\bar{w}, \bar{\al})\\
=& \left(\bar{w}+(-\Delta+\la_i)^{-1}
\left(|\bb|t_j(\vec{u})u_j^2 \bar{w}-\bar{\al} \mu_i t_i(\vec{u})u_i^3\right),\,\,\int_{\Om}u_i^3\bar{w}\,dx\right).
\end{align*}
}%
If $\Phi(\bar{w}, \bar{\al})=(0, 0)$, then we multiply the equation
$$-\Delta \bar{w}+\la_i \bar{w}+|\bb|t_j(\vec{u})u_j^2 \bar{w}=\bar{\al} \mu_i t_i(\vec{u})u_i^3$$
by $\bar{w}$ and obtain
$$\|\bar{w}\|_{\la_i}^2\le \bar{\al} \mu_i t_i(\vec{u})\int_{\Om}u_i^3\bar{w}\,dx=0.$$
So $\bar{w}=0$ and then $\bar{\al} \mu_i t_i(\vec{u})u_i^3\equiv 0$ in $\Om$. Since $\mu_i>0, t_i(\vec{u})>0$ and $|u_i|_4\ge 1/2$, we see that $\bar{\al}=0$.
Hence $\Phi$ is injective.

On the other hand, for any $(f, c)\in H_0^1(\Om)\times \R$, let $v_1, v_2\in H_0^1(\Om)$ be solutions of the linear problems
{\allowdisplaybreaks
\begin{align*}
&-\Delta v_1+\la_i v_1+|\bb|t_j(\vec{u})u_j^2 v_1=(-\Delta +\la_i) f,\\
&-\Delta v_2+\la_i v_2+|\bb|t_j(\vec{u})u_j^2 v_2=\mu_i t_i(\vec{u})u_i^3.
\end{align*}
}%
Since $|u_i|_4>1/2$, so $v_2\neq 0$ and then $\int_{\Om}u_i^3 v_2\,dx>0$.
Let $\al_0=(c-\int_{\Om}u_i^3 v_1\,dx)/\int_{\Om}u_i^3 v_2\,dx$, then
$\Phi(v_1+\al_0 v_2, \al_0)=(f, c)$. Hence $\Phi$ is surjective, that is, $\Phi$ is a bijective map. This completes the proof.\end{proof}

\bl\label{lemma4}Assume that $\{\vec{u}_n=(u_{n, 1}, u_{n, 2}) : n\ge 1\}\subset \mathcal{M}_\bb$ is bounded in $H$
and $\vec{u}_n\rightharpoonup \vec{u}=(u_1, u_2)\in \mathcal{M}_\bb$ weakly in $H$. Then there exists $\vec{w}\in H$ such that, up to a subsequence,
$\vec{w}_n:=K(\vec{u}_n)\to \vec{w}$ strongly in $H$.\el

\begin{proof} Recall the definition of $\mathcal{M}_\bb$ in (\ref{eq2-8}),
we deduce from (\ref{eq2-5}) and (\ref{eq2-10}) that there exists $C_0>0$ independent of $\vec{u}\in\mathcal{M}_\bb$ such that
\begin{align}\label{eq2-16}
t_i(\vec{u})=\frac{\mu_j\|u_i\|_{\la_i}^2+|\bb|\|u_j\|_{\la_j}^2\int_{\Om}u_1^2u_2^2}{\mu_1\mu_2-|\bb|^2(\int_{\Om}u_1^2 u_2^2)^2}
\ge \frac{1}{\mu_i}\|u_i\|_{\la_i}^2\ge C_0, \quad\forall\, \vec{u}\in\mathcal{M}_\bb.
\end{align}

Since $\vec{u}_n\rightharpoonup \vec{u}=(u_1, u_2)\in \mathcal{M}_\bb$ weakly in $H$, so up to a subsequence,
$u_{n, i}\to u_i$ strongly in $L^{4}(\Om)$. Then
$$\lim_{n\to\iy}\left(\mu_1\mu_2-|\bb|^2\left(\int_{\Om}u_{n,1}^2 u_{n, 2}^2\right)^2\right)=\mu_1\mu_2-|\bb|^2\left(\int_{\Om}u_1^2 u_2^2\right)^2>0,$$
where the assumption $\vec{u}\in \mathcal{M}_\bb$ is used. Hence we may assume that $t_i(\vec{u}_n)$ are uniformly bounded for any $n\ge 1$ and $i=1, 2$, and up to a subsequence, $t_i(\vec{u}_n)\to t_i>0$.
Recall that $w_{n, i}=\al_{n, i}\tilde{w}_{n, i}$, where $\al_{n, i}$ and $\tilde{w}_{n, i}$ are seen in (\ref{eq2-012})-(\ref{eq2-12}).
By (\ref{eq2-012}) we have
$$\|\tilde{w}_{n, i}\|_{\la_i}^2\le \mu_i t_i(\vec{u}_n)\int_{\Om}u_{n, i}^3 \tilde{w}_{n, i}\,dx\le C|\tilde{w}_{n, i}|_4\le C\|\tilde{w}_{n, i}\|_{\la_i},$$
which implies that $\{\tilde{w}_{n, i} : n\ge 1\}$ are bounded in $H_0^1(\Om)$. Up to a subsequence, we may assume that
$\tilde{w}_{n, i}\to \tilde{w}_i$ weakly in $H_0^1(\Om)$ and strongly in $L^4(\Om)$.
Then by (\ref{eq2-012}) and H\"{o}lder inequality we get
{\allowdisplaybreaks
\begin{align*}
&\int_{\Om}\nabla\tilde{w}_{n, i}\nabla(\tilde{w}_{n, i}-\tilde{w}_i)\,dx+\la_i\int_{\Om}\tilde{w}_{n, i}(\tilde{w}_{n, i}-\tilde{w}_i)\,dx\\
=&-|\bb|t_j(\vec{u}_n)\int_{\Om}u_{n, j}^2\tilde{w}_{n, i}(\tilde{w}_{n, i}-\tilde{w}_i)+\mu_i t_i(\vec{u}_n)\int_{\Om}u_{n, i}^3(\tilde{w}_{n, i}-\tilde{w}_i)\,dx\to 0
\end{align*}
}%
as $n\to\iy$. Hence
\be\label{eq2-34}\|\tilde{w}_{n, i}\|_{\la_i}^2=\int_{\Om}(\nabla\tilde{w}_{n, i}\nabla\tilde{w}_i+\la_i \tilde{w}_{n, i}\tilde{w}_i)+o(1)=\|\tilde{w}_i\|_{\la_i}^2+o(1),\ee
that is, $\tilde{w}_{n, i}\to \tilde{w}_i$ strongly in $H_0^1(\Om)$. Again by (\ref{eq2-012}) we know that $\tilde{w}_i$ satisfies
$$-\Delta \tilde{w}_i+\la_i \tilde{w}_i+|\bb|t_ju_j^2 \tilde{w}_i=\mu_i t_i u_i^3.$$
Since $|u_i|_4=1$, so $\tilde{w}_i\neq 0$ and then $\int_{\Om}u_i^3 \tilde{w}_i\,dx>0$, which implies that
$$\lim_{n\to\iy}\al_{n, i}=\lim_{n\to\iy}\frac{1}{\int_{\Om}u_{n,i}^3 \tilde{w}_{n,i}}=\frac{1}{\int_{\Om}u_i^3 \tilde{w}_i}=:\al_i.$$
Therefore, $w_{n, i}=\al_{n, i}\tilde{w}_{n, i}\to \al_i\tilde{w_i}=: w_i$ strongly in $H_0^1(\Om)$.\end{proof}

To continue our proof, we need to use {\it vector genus} introduced by \cite{TT} to define proper minimax energy levels.
Recall (\ref{involution}) and (\ref{eq2-8}), as in \cite{TT} we consider the class of sets
$$\mathcal{F}=\{A\subset \mathcal{M} : A\,\,\hbox{ is closed and}\,\,\sg_i(\vec{u})\in A\,\,\forall\,\vec{u}\in A,\,\,i=1, 2\},$$
and, for each $A\in \mathcal{F}$ and $k_1, k_2\in\mathbb{N}$, the class of functions
$$F_{(k_1, k_2)}(A)=\left\{f=(f_1, f_2): A\to \prod_{i=1}^2\R^{k_i-1} : \begin{array}{lll}  f_i : A\to \R^{k_i-1}\,\,\hbox{continuous,}\\
 f_i(\sg_i(\vec{u}))=-f_i(\vec{u})\,\,\hbox{for each}\,\,i,\\
 f_i(\sg_j(\vec{u}))=f_i(\vec{u})\,\,\hbox{for}\,\,j\neq i
\end{array} \right\}.$$
Here, we denote $\R^0:=\{0\}$. Let us recall vector genus from \cite{TT}.
\begin{definition}\label{definition1} (Vector genus, see \cite{TT})
Let $A\in \mathcal{F}$ and take any $k_1, k_2\in\mathbb{N}$. We say that $\vec{\ga}(A)\ge (k_1, k_2)$ if for every $f\in F_{(k_1, k_2)}(A)$ there exists $\vec{u}\in A$ such that $f(\vec{u})=(f_1(\vec{u}), f_2(\vec{u}))=(0, 0)$. We denote
$$\Gamma^{(k_1, k_2)}:=\{A\in\mathcal{F} : \vec{\ga}(A)\ge (k_1, k_2)\}.$$
\end{definition}

\bl\label{lemma5}(see \cite{TT}) With the previous notations, the following properties hold.
\begin{itemize}
\item[$(i)$] Take $A_1\times A_2\subset \mathcal{M}$ and let $\eta_i: S^{k_i-1}:=\{x\in \R^{k_i} : |x|=1\}\to A_i$ be a homeomorphism such that
$\eta_i(-x)=-\eta_i(x)$ for every $x\in S^{k_i-1}$, $i=1, 2$. Then $A_1\times A_2\in \Gamma^{(k_1, k_2)}$.

\item[$(ii)$] We have $\overline{\eta(A)}\in \Gamma^{(k_1, k_2)}$ whenever $A\in \Gamma^{(k_1, k_2)}$ and a continuous map $\eta : A\to \mathcal{M}$ is such
that $\eta\circ \sg_i=\sg_i\circ\eta,\,\,\forall\,i=1, 2$.
\end{itemize}\el

To obtain sign-changing solutions, as in many references such as \cite{CMT, BLT, Zou}, we should use cones of positive functions.
Precisely, we define
\be\label{cone}\mathcal{P}_i:=\{\vec{u}=(u_1, u_2)\in H : u_i\ge 0\},\quad \mathcal{P}:=\bigcup_{i=1}^2 (\mathcal{P}_i\cup -\mathcal{P}_i).\ee
Moreover, for $\dd>0$ we define
$\mathcal{P}_\dd:=\{\vec{u}\in H \,:\,\hbox{dist}_4(\vec{u}, \mathcal{P})<\dd \}$,
where
{\allowdisplaybreaks
\begin{align}\label{cone1}&\hbox{dist}_4(\vec{u}, \mathcal{P}):=\min\big\{\hbox{dist}_4(u_i,\,\mathcal{P}_i),\,\,\hbox{dist}_4(u_i,\,-\mathcal{P}_i),\quad i=1, 2\big\},\\
&\hbox{dist}_4(u_i,\,\pm\mathcal{P}_i ):=\inf\{|u_i-v|_4 \,\,:\,\, v\in \pm\mathcal{P}_i\}.\nonumber
\end{align}
}%
Denote $u^{\pm}:=\max\{0, \pm u\}$, then it is easy to check that $\hbox{dist}_4(u_i, \pm\mathcal{P}_i)=|u_i^{\mp}|_4$.

\bl\label{lemma6}Let $k_1, k_2\ge 2$. Then for any $\dd<2^{-1/4}$ and any $A\in\Gamma^{(k_1, k_2)}$ there holds $A\setminus \mathcal{P}_\dd\neq\emptyset$.\el

\begin{proof} Fix any $A\in\Gamma^{(k_1, k_2)}$. Consider
\be\label{eq2-36} f=(f_1, f_2) : A\to \R^{k_1-1}\times\R^{k_2-1},\quad f_i(\vec{u})=\left(\int_{\Om}|u_i|^3u_i\,dx, 0,\cdots, 0\right).\ee
Clearly $f\in F_{(k_1, k_2)}(A)$, so there exists $\vec{u}\in A$ such that $f(\vec{u})=0$. Note that $\vec{u}\in A\subset \mathcal{M}$, we conclude that
$$\int_{\Om}(u_i^{+})^4\,dx=\int_{\Om}(u_i^-)^4\,dx=1/2,\quad\hbox{for $i=1,2$},$$
that is, $\hbox{dist}_4(\vec{u}, \mathcal{P})=2^{-1/4}$, and so $\vec{u}\in A\setminus \mathcal{P}_\dd$ for every $\dd<2^{-1/4}$.\end{proof}

\bl\label{lemma7}  There exist $A\in \Gamma^{(k_1, k_2)}$ and a positive constant $c^{k_1, k_2}\in\mathbb{N}$ independent of $\bb<0$ such that
$\sup_A J_\bb\le c^{k_1, k_2}$ for any $\bb<0$. \el

\begin{proof} Take nonempty open subsets $B_1, B_2\subset\Om$ such that $B_1\cap B_2=\emptyset$.
Let $\{\vp^i_{k} : 1\le k\le k_i\}\subset H_0^1(B_i)$ be linearly independent subsets, and define
$$A_i:=\big\{u\in \hbox{span}\{\vp^i_{1},\cdots, \vp_{k_i}^i\} \,:\, |u|_4=1\big\}.$$
Clearly there exists an odd homeomorphism from $S^{k_i-1}$ to $A_i$. By Lemma \ref{lemma5}-$(i)$ one has $A:=A_1\times A_2\in \Gamma^{(k_1, k_2)}$.
For any $\vec{u}=(u_1, u_2)\in A$, since $u_i\in H_0^1(B_i)$, so $u_1\cdot u_2\equiv 0$, which implies $\vec{u}\in\mathcal{M}_\bb$ and
$$J_\bb (\vec{u})=\frac{1}{4\mu_1\mu_2}(\mu_2\|u_1\|_{\la_1}^4+
\mu_1\|u_2\|_{\la_2}^4).$$
Since all norms of a finite dimensional linear space are equivalent, so there exists $C_{k_i}>0$ such that
$\|u_i\|_{\la_i}\le C_{k_i}|u_i|_4=C_{k_i}$ for any $u_i\in A_i$.
Hence there exists $c^{k_1, k_2}\in\mathbb{N}$ independent of $\bb<0$ such that
$\sup_{A}J_\bb(\vec{u})\le c^{k_1, k_2}$ for any $\bb<0$.
\end{proof}

For every $k_1, k_2\ge 2$ and $0<\dd < 2^{-1/4}$, we define
\be\label{eq2-17}c_{\bb,\dd}^{k_1, k_2}:=\inf_{A\in \Gamma_\bb^{(k_1, k_2)}}\sup_{\vec{u}\in A\setminus \mathcal{P}_\dd}J_\bb(\vec{u}),\ee
where
\be\label{eq2-17-1}\Gamma_\bb^{(k_1, k_2)}:=\left\{A\in \Gamma^{(k_1, k_2)} \,:\, \sup_{A}J_\bb< c^{k_1, k_2}+1\right\}.\ee
Lemma \ref{lemma7} yields $\Gamma_\bb^{(k_1, k_2)}\neq \emptyset$ and so $c_{\bb, \dd}^{k_1, k_2}$ is well defined. Moreover,
$$c_{\bb, \dd}^{k_1, k_2}\le c^{k_1, k_2}\quad\hbox{for every}\,\,\bb<0\,\,\hbox{and}\,\,\dd>0. $$
Recall that $\inf_{\mathcal{M}}J_\bb\ge C_1$, so $c_{\bb, \dd}^{k_1, k_2}\ge C_1>0$.
We will prove that $c_{\bb, \dd}^{k_1, k_2}$ is a critical value of $E_\bb$ provided that $\dd>0$ is sufficiently small.
As we will see in Remark \ref{rmk4}, we can not replace $\Gamma_\bb^{(k_1, k_2)}$ by $\Gamma^{(k_1, k_2)}$ in the definition of $c_{\bb,\dd}^{k_1, k_2}$.

\bl\label{lemma8}For any sufficiently small $\dd\in (0, 2^{-1/4})$, there holds
$$\hbox{dist}_4(K(\vec{u}), \mathcal{P})<\dd/2,\quad\forall\,\vec{u}\in\mathcal{M},\,J_\bb(\vec{u})\le c^{k_1, k_2}+1,\,
\hbox{dist}_4(\vec{u}, \mathcal{P})<\dd.$$\el

\begin{proof} Assume by contradiction that there exist $\dd_n\to 0$ and $\vec{u}_n=(u_{n, 1}, u_{n, 2})\in\mathcal{M}$
such that $J_\bb (\vec{u}_n)\le c^{k_1, k_2}+1$,
$\hbox{dist}_4(\vec{u}_n, \mathcal{P})<\dd_n$ and $\hbox{dist}_4(K(\vec{u}_n), \mathcal{P})\ge\dd_n/2$.
Without loss of generality we may assume that $\hbox{dist}_4(\vec{u}_n, \mathcal{P})=\hbox{dist}_4(u_{n, 1}, \mathcal{P}_1)$.
 Recall the definition of $J_\bb$, we see that
$\vec{u}_n\in \mathcal{M}_\bb$ and
\be\label{eq2-18}c^{k_1, k_2}+1\ge J_\bb (\vec{u}_n)\ge\frac{1}{4}\frac{\mu_2\|u_{n,1}\|_{\la_1}^4
+\mu_1\|u_{n, 2}\|_{\la_2}^4}{\mu_1\mu_2-|\bb|^2\left(\int_{\Om}u_{n,1}^2 u_{n,2}^2\right)^2}.\ee
This implies that $\vec{u}_n$ are uniformly bounded in $H$. Up to a subsequence, we may assume that $\vec{u}_n\to \vec{u}=(u_1, u_2)$ weakly in $H$ and strongly in $L^4(\Om)\times L^4(\Om)$. Hence $|u_i|_4=1$ and $\vec{u}\in \mathcal{M}$. Moreover, since (\ref{eq2-10})
yields $\|u_{n, i}\|_{\la_i}^2\ge C>0$, where $C$ is independent of $n$, so we deduce from (\ref{eq2-18}) that
$$\mu_1\mu_2-|\bb|^2\left(\int_{\Om}u_1^2u_2^2\right)^2=\lim_{n\to\iy}\left(\mu_1\mu_2-|\bb|^2\left(\int_{\Om}u_{n,1}^2 u_{n, 2}^2\right)^2\right)> 0,$$
that is, $\vec{u}\in\mathcal{M}_\bb$. Write $K(\vec{u}_n)=\vec{w}_n=(w_{n, 1}, w_{n, 2})$ and $w_{n, i}=\al_{n, i}\tilde{w}_{n, i}$ as in
the proof of Lemma \ref{lemma4}. Then by the proof of Lemma \ref{lemma4}, we see that $t_i(\vec{u}_n)$ and $\al_{n, i}$ are all uniformly bounded.
Combining this with (\ref{eq2-13}), we deduce that
{\allowdisplaybreaks
\begin{align*}
\hbox{dist}_4(w_{n, 1}, \mathcal{P}_1)|w_{n, 1}^-|_4 &=|w_{n, 1}^-|_4^2\le C\int_{\Om}|\nabla w_{n, 1}^-|^2+\la_1 (w_{n, 1}^-)^2\,dx\\
&\le C\int_{\Om}\left(|\nabla w_{n, 1}^-|^2+\la_1 (w_{n, 1}^-)^2+|\bb|t_2(\vec{u}_n)u_{n, 2}^2 (w_{n, 1}^-)^2\right)\\
&=-C\al_{n, 1}\mu_1 t_1(\vec{u}_n)\int_{\Om}u_{n, 1}^3 w_{n, 1}^-\,dx\\
&\le C\int_{\Om}(u_{n, 1}^-)^3 w_{n, 1}^-\,dx\le C|u_{n, 1}^-|_4^3|w_{n, 1}^-|_4\\
&=C\hbox{dist}_4(u_{n, 1}, \mathcal{P}_1)^3 |w_{n, 1}^-|_4
\le C \dd_n^3|w_{n, 1}^-|_4.
\end{align*}
}%
So $\hbox{dist}_4(K(\vec{u}_n), \mathcal{P})\le \hbox{dist}_4(w_{n, 1}, \mathcal{P}_1)\le C\dd_n^3<\dd_n/2$
holds for $n$ sufficiently large, which is a contradiction. This completes the proof.\end{proof}

Now let us define a map
$$V: \mathcal{M}_\bb^\ast \to H,\quad V(\vec{u}):=\vec{u}-K(\vec{u}).$$
We will prove that $(\sqrt{t_1(\vec{u})}u_1, \sqrt{t_2(\vec{u})}u_2)$ is a nontrivial solution of (\ref{eq2}), if $\vec{u}=(u_1, u_2)\in\mathcal{M}_\bb$ satisfies $V(\vec{u})=0$.

\bl\label{lemma9} Let $\vec{u}_n=(u_{n, 1}, u_{n, 2})\in\mathcal{M}_\bb$ be such that
$$J_\bb(\vec{u}_n)\to c<\iy\quad\hbox{and}\quad V(\vec{u}_n)\to 0\quad\hbox{strongly in $H$}.$$
Then up to a subsequence, there exists $\vec{u}\in \mathcal{M}_\bb$ such that $\vec{u}_n\to \vec{u}$ strongly in $H$ and $V(\vec{u})=0$. \el

\begin{proof} Without loss of generality we may assume that $J_\bb(\vec{u}_n)\le c+1$ for all $n\ge 1$. Then by the proof of Lemma \ref{lemma8}, up to a subsequence, we may assume that $\vec{u}_n\rightharpoonup \vec{u}=(u_1, u_2)\in\mathcal{M}_\bb$ weakly in $H$. By Lemma \ref{lemma4},
there exists $\vec{w}\in H$ such that, up to a subsequence,
$\vec{w}_n:=K(\vec{u}_n)=(w_{n, 1}, w_{n, 2})\to \vec{w}=(w_1, w_2)$ strongly in $H$. Recall $V(\vec{u}_n)\to 0$, we get
{\allowdisplaybreaks
\begin{align*}
\int_{\Om}\nabla u_{n, i}\nabla (u_{n, i}-u_i)=&\int_{\Om}\nabla (w_{n, i}-w_{i})\nabla (u_{n, i}-u_i)+\int_{\Om}\nabla w_{i}\nabla (u_{n, i}-u_i)\\
&+\int_{\Om}\nabla (u_{n, i}-w_{n, i})\nabla (u_{n, i}-u_i)=o(1).
\end{align*}
}%
Then similarly as (\ref{eq2-34}) we see that $\vec{u}_n\to \vec{u}$ strongly in $H$. By Lemma \ref{lemma3}
we have $V(\vec{u})=\lim_{n\to\iy} V(\vec{u}_n)=0$.\end{proof}

\bl\label{lemma10}Recall $C_0>0$ in (\ref{eq2-16}). Then
$$J_\bb'(\vec{u})[V(\vec{u})]\ge C_0\|V(\vec{u})\|_H^2,\quad \hbox{for any}\,\, \vec{u}\in\mathcal{M}_\bb.$$
\el

\begin{proof} Fix any $\vec{u}=(u_1, u_2)\in\mathcal{M}_\bb$, write $\vec{w}=K(\vec{u})=(w_1, w_2)$ as above,
then $V(\vec{u})=(u_1-w_1, u_2-w_2)$. By (\ref{eq2-13}) we have
$\int_{\Om}u_i^3(u_i-w_i)\,dx=1-1=0$.
Then we deduce from (\ref{eq2-11})-(\ref{eq2-11-1}), (\ref{eq2-13}) and (\ref{eq2-16}) that
{\allowdisplaybreaks
\begin{align*}
&J_\bb'(\vec{u})[V(\vec{u})]\\
=&\sum_{i=1}^2t_i(\vec{u})\int_{\Om}\left(\nabla u_i\nabla (u_i-w_i)+\la_i u_i(u_i-w_i)+t_j(\vec{u})|\bb|u_i(u_i-w_i)u_j^2\right)\,dx\\
\ge &\sum_{i=1}^2t_i(\vec{u})\int_{\Om}\left(\nabla u_i\nabla (u_i-w_i)+\la_i u_i(u_i-w_i)+t_j(\vec{u})|\bb|w_i(u_i-w_i)u_j^2\right)\,dx\\
=&\sum_{i=1}^2t_i(\vec{u})\int_{\Om}\big(\nabla u_i\nabla (u_i-w_i)+\la_i u_i(u_i-w_i)-\nabla w_i\nabla(u_i-w_i)\\
&\qquad\qquad\qquad-\la_i w_i(u_i-w_i)+\al_i\mu_i t_i(\vec{u})u_i^3(u_i-w_i)\big)\,dx\\
=&\sum_{i=1}^2t_i(\vec{u})\int_{\Om}|\nabla (u_i-w_i)|^2+\la_i |u_i-w_i|^2\,dx\ge C_0\|V(\vec{u})\|_H^2.
\end{align*}
}%
This completes the proof.\end{proof}

\bl\label{lemma11} There exists a unique global solution $\eta=(\eta_1, \eta_2) : [0, \iy)\times \mathcal{M}_\bb \to H$ for the initial value problem
\be\label{eq2-19}\frac{d}{dt}\eta(t,\vec{u})=-V(\eta(t, \vec{u})), \quad \eta(0, \vec{u})=\vec{u}\in\mathcal{M}_\bb.\ee
Moreover,

\begin{itemize}
\item[$(i)$] $\eta(t, \vec{u})\in \mathcal{M}_\bb$ for any $ t>0$ and $u\in\mathcal{M}_\bb$.
\item[$(ii)$] $\eta(t, \sg_i(\vec{u}))=\sg_i(\eta(t, \vec{u}))$ for any $t>0$, $u\in\mathcal{M}_\bb$ and $i=1, 2$.
\item[$(iii)$] For every $\vec{u}\in\mathcal{M}_\bb$, the map $t\mapsto J_\bb(\eta(t, \vec{u}))$ is non-increasing.
\item[$(iv)$] There exists $\dd_0\in (0, 2^{-1/4})$ such that, for every $\dd<\dd_0$, there holds
$$\eta(t, \vec{u})\in \mathcal{P}_\dd \quad\hbox{whenever}\,\,u\in \mathcal{M}_\bb\cap \mathcal{P}_\dd, \, J_\bb(u)\le c^{k_1, k_2}+1\,\,\hbox{and}\,\,t>0.$$
\end{itemize}

\el

\begin{proof} Recalling Lemma \ref{lemma3}, one has $V(\vec{u})\in C^1(\mathcal{M}_\bb^\ast, H)$.
Since $\mathcal{M}_\bb\subset\mathcal{M}_\bb^\ast$ and $\mathcal{M}_\bb^\ast$ is open, so (\ref{eq2-19}) has a unique solution
$\eta: [0, T_{\max})\times\mathcal{M}_\bb\to H$,
where $T_{\max}>0$ is the maximal time such that $\eta(t, \vec{u})\in \mathcal{M}_\bb^\ast$ for all
$t\in [0, T_{\max})$ (Note that $V(\cdot)$ is defined only on $\mathcal{M}_\bb^\ast$). We should prove $T_{\max}=+\iy$ for any $\vec{u}\in \mathcal{M}_\bb$.
Fix any $\vec{u}=(u_1, u_2)\in \mathcal{M}_\bb$, one has
{\allowdisplaybreaks
\begin{align*}
\frac{d}{dt}\int_{\Om}\eta_i(t, \vec{u})^4\,dx &=-4\int_{\Om}\eta_i(t, \vec{u})^3(\eta_i(t, \vec{u})-K_i(\eta(t, \vec{u})))\,dx\\
&=4-4\int_{\Om}\eta_i(t, \vec{u})^4\,dx,\quad \forall\, 0< t<T_{\max},
\end{align*}
}%
that is
$$\frac{d}{dt}\left[e^{4t}\left(\int_{\Om}\eta_i(t, \vec{u})^4\,dx-1\right)\right]=0.$$
Recalling $\int_{\Om}\eta_i(0, \vec{u})^4\,dx=\int_{\Om}u_i^4\,dx=1$, we see that
$$\int_{\Om}\eta_i(t, \vec{u})^4\,dx\equiv 1
\quad \hbox{for all}\,\, 0\le t<T_{\max}.$$
So $\eta(t, \vec{u})\in \mathcal{M}$, that is $\eta(t, \vec{u})\in\mathcal{M}\cap\mathcal{M}_\bb^\ast=\mathcal{M}_\bb$ for all $t\in [0, T_{\max})$.
Assume by contradiction that $T_{\max}<+\iy$, then either $\eta(T_{\max}, \vec{u})\in\mathcal{M}\setminus\mathcal{M}_\bb^\ast$ or
$\lim_{t\to T_{\max}}\|\eta(t, \vec{u})\|_H=+\iy$. If $\eta(T_{\max}, \vec{u})\in\mathcal{M}\setminus\mathcal{M}_\bb^\ast$,
then the definition of $J_\bb$ yields
$J_\bb (\eta(T_{\max}, \vec{u}))=+\iy$.
Since $\eta(t,\vec{u})\in \mathcal{M}_\bb$ for any $t\in [0, T_{\max})$, we deduce from Lemma \ref{lemma10} that
{\allowdisplaybreaks
\begin{align}\label{eq2-20}
J_\bb\left(\eta\left(T_{\max}, \vec{u}\right)\right)&=J_\bb(\eta(0, \vec{u}))+\int_0^{T_{\max}}\frac{d}{dt}J_\bb(\eta(t, \vec{u}))\,dt\nonumber\\
&=J_\bb(\vec{u})-\int_0^{T_{\max}} J_\bb'(\eta(t, \vec{u}))[V(\eta(t, \vec{u}))]\,dt\nonumber\\
&\le J_\bb(\vec{u})-C_0\int_0^{T_{\max}} \|V(\eta(t, \vec{u}))\|_H^2\,dt\le J_\bb(\vec{u})<+\iy,
\end{align}
}%
a contradiction. So $\lim_{t\to T_{\max}}\|\eta(t, \vec{u})\|_H=+\iy$. Similarly as (\ref{eq2-20}), we see that $J_\bb(\eta(t, \vec{u}))\le J_\bb(\vec{u})<+\iy$ for all $t\in [0, T_{\max})$,
and so
$$\frac{1}{4\mu_1\mu_2}\left(\mu_2\|\eta_1(t, \vec{u})\|_{\la_1}^4+
\mu_1\|\eta_2(t, \vec{u})\|_{\la_2}^4\right)\le J_\bb (\eta(t, \vec{u}))\le J_\bb(\vec{u})<+\iy,$$
which means that $\|\eta(t, \vec{u})\|_H^2$ are uniformly bounded for all $[0, T_{\max})$, also a contradiction. Hence $T_{\max}=+\iy$ and $(i), (iii)$ hold.

By (\ref{eq2-15}) we have $V(\sg_i(\vec{u}))=\sg_i(V(\vec{u}))$. Then by the uniqueness of solutions of the initial value
problem (\ref{eq2-19}), it is easy to check that
$(ii)$ holds.

Finally, let $\dd_0\in (0, 2^{-1/4})$ such that Lemma \ref{lemma8} holds for every $\dd<\dd_0$. For any $\vec{u}\in\mathcal{M}_\bb$ with
$J_\bb (\vec{u})\le c^{k_1, k_2}+1$ and $\hbox{dist}_4(\vec{u}, \mathcal{P})=\dd<\dd_0$, since
$$\eta(t, \vec{u})=\vec{u}+t\frac{d}{dt}{\eta}(0, \vec{u})+o(t)=\vec{u}-t V(\vec{u})+o(t)=(1-t)\vec{u}+tK(\vec{u})+o(t),$$
so we see from Lemma \ref{lemma8} that
{\allowdisplaybreaks
\begin{align*}
\hbox{dist}_4(\eta(t, \vec{u}), \mathcal{P})&=\hbox{dist}_4((1-t)\vec{u}+tK(\vec{u})+o(t), \mathcal{P})\\
&\le(1-t)\hbox{dist}_4(\vec{u}, \mathcal{P})+t\hbox{dist}_4(K(\vec{u}), \mathcal{P})+o(t)\\
&\le(1-t)\dd+t\dd/2+o(t)<\dd
\end{align*}
}%
for $t>0$ sufficiently small. Hence $(iv)$ holds.\end{proof}

Now we are in a position to prove Theorem \ref{th1}.

\begin{proof}[Proof of Theorem \ref{th1}.]
{\bf Step 1.} Take any $\dd\in (0, \dd_0)$.  We prove that (\ref{eq2}) has a sign-changing solution $(\tilde{u}_1, \tilde{u}_2)\in H$ such that
$E_\bb (\tilde{u}_1, \tilde{u}_2)=c_{\bb, \dd}^{k_1, k_2}$.

Write $c_{\bb, \dd}^{k_1, k_2}$ simply by $c$ in this step. We claim that
there exists a sequence $\{\vec{u}_n :n\ge 1\}\subset \mathcal{M}_\bb$ such that
\be\label{eq2-21}J_\bb(\vec{u}_n)\to c,\,\, V(\vec{u}_n)\to 0\,\,\,\hbox{as $n\to\iy$,}\,\,\,
\hbox{and}\,\, \hbox{dist}_4(\vec{u}_n, \mathcal{P})\ge \dd,\,\,\,\forall\,n\in\mathbb{N}.\ee

If (\ref{eq2-21}) does not hold, there exists small $\e\in (0, 1)$ such that
$$\|V(\vec{u})\|_H^2\ge \e,\quad\forall\,u\in\mathcal{M}_\bb,\,\,|J_\bb(\vec{u})-c|\le 2\e,\,\,\hbox{dist}_4(\vec{u}, \mathcal{P})\ge \dd.$$
Recall the definition of $c$ in (\ref{eq2-17}), there exists $A\in \Gamma_\bb^{(k_1, k_2)}$ such that
$$\sup_{A\setminus \mathcal{P}_\dd}J_\bb<c+\e.$$
Since $\sup_{A} J_\bb<c^{k_1, k_2}+1$, so $A\subset \mathcal{M}_\bb$. Then we can consider $B=\eta(2/C_0, A)$, where $\eta$ is in Lemma \ref{lemma11} and $C_0$ is in (\ref{eq2-16}). By Lemma \ref{lemma5}-$(ii)$ and Lemma \ref{lemma11}-$(ii)$ we have $B\in \Gamma^{(k_1, k_2)}$. Again by Lemma \ref{lemma11}-$(iii)$, we have
$\sup_{B} J_\bb\le\sup_A J_\bb<c^{k_1, k_2}+1$, that is $B\in \Gamma_\bb^{(k_1, k_2)}$ and so
$\sup_{B\setminus \mathcal{P}_\dd}J_\bb\ge c$.
Then by Lemma \ref{lemma6} we can take $\vec{u}\in A$ such that $\eta(2/C_0, \vec{u})\in B\setminus \mathcal{P}_\dd$ and
$$c-\e\le \sup_{B\setminus \mathcal{P}_\dd}J_\bb-\e< J_\bb (\eta(2/C_0, \vec{u})).$$
Since $J_{\bb}(\eta(t, \vec{u}))\le J_\bb(\vec{u})<c^{k_1, k_2}+1$ for any $t>0$,
Lemma \ref{lemma11}-$(iv)$ yields $\eta(t, \vec{u})\not\in\mathcal{P}_\dd$ for any $t\in [0, 2/C_0]$.
In particular, $\vec{u}\not\in\mathcal{P}_\dd$ and so $J_\bb(\vec{u})< c+\e$. Then
for any $t\in [0, 2/C_0]$, we have
$$c-\e< J_\bb (\eta(2/C_0, \vec{u}))\le J_\bb (\eta(t, \vec{u}))\le J_\bb (\vec{u})< c+\e,$$
which implies $\|V(\eta(t, \vec{u}))\|_H^2\ge \e$ and
\begin{align*}
\frac{d}{dt}J_\bb (\eta(t, \vec{u}))=-J_\bb'(\eta(t, \vec{u}))[V(\eta(t, \vec{u}))]\le-C_0\|V(\eta(t, \vec{u}))\|_H^2\le-C_0 \e
\end{align*}
for every $t\in [0, 2/C_0]$. Hence,
$$c- \e< J_\bb (\eta(2/C_0, \vec{u}))\le J_\bb(\vec{u})-\int_{0}^{2/C_0}C_0\e\,dt<c+\e-2\e=c-\e,$$
a contradiction. Therefore (\ref{eq2-21}) holds, and by Lemma \ref{lemma9}, up to a subsequence,
there exists $\vec{u}=(u_1, u_2)\in \mathcal{M}_\bb$ such that $\vec{u}_n\to \vec{u}$ strongly in $H$ and $V(\vec{u})=0$, $J_\bb(\vec{u})=c=c_{\bb, \dd}^{k_1, k_2}$. Since $\hbox{dist}_4(\vec{u}_n, \mathcal{P})\ge \dd$, so
$\hbox{dist}_4(\vec{u}, \mathcal{P})\ge \dd$, which implies that both $u_1$ and $u_2$ are sign-changing.

Since $V(\vec{u})=0$, so $\vec{u}=K(\vec{u})$. Combining this with (\ref{eq2-13})-(\ref{eq2-14}), we see that $\vec{u}$ satisfies
\be\label{eq2-22-1}\begin{cases}
-\Delta u_1+\la_1 u_1+|\bb|t_2(\vec{u})u_2^2 u_1=\al_1\mu_1 t_1(\vec{u})u_1^3,\\
-\Delta u_2+\la_2 u_2+|\bb|t_1(\vec{u})u_1^2 u_2=\al_2\mu_2 t_2(\vec{u})u_2^3.
\end{cases}\ee
Recall that $|u_i|_4=1$ and $t_i(\vec{u})$ satisfies (\ref{eq2-4}). Multiplying (\ref{eq2-22-1}) by $u_i$ and integrating over $\Om$, we obtain that $\al_1=\al_2=1$.
Again by (\ref{eq2-22-1}), we see that $(\tilde{u}_1, \tilde{u}_2):=(\sqrt{t_1(\vec{u})}u_1, \sqrt{t_2(\vec{u})}u_2)$ is a sign-changing solution of the original problem (\ref{eq2}). Moreover, (\ref{eq2-6}) and (\ref{eq2-9}) yield
$$E_\bb (\tilde{u}_1, \tilde{u}_2)=J_\bb (u_1, u_2)=c_{\bb, \dd}^{k_1, k_2}.$$
This completes the proof of Step 1.

{\bf Step 2.} We prove that (\ref{eq2}) has infinitely many sign-changing solutions $(u_{n, 1}, u_{n, 2})$ such that
\be\label{eq2-24}\|u_{n, 1}\|_{L^\iy(\Om)}+\|u_{n,2}\|_{L^\iy(\Om)}\to+\iy\quad\hbox{as}\,\,\,n\to+\iy.\ee

It suffices to prove that
\be\label{eq2-23}\lim_{k_1\to\iy}\inf_{0<\dd\le 2^{-5/4}}c_{\bb, \dd}^{k_1, k_2}=+\iy.\ee

Assume by contradiction that there exist $k_1^n\to\iy$, $\dd_n\in (0, 2^{-5/4}]$ and a positive constant $C$ such that
$c_{\bb, \dd_n}^{k_1^n, k_2}\le C$ for every $n\in\mathbb{N}$.
Then there exists $A_n\in \Gamma_\bb^{(k_1^n, k_2)}$ such that
$$\sup_{A_n\setminus\mathcal{P}_{\dd_n}} J_\bb\le C+1,\quad\forall\,n\in\mathbb{N}.$$

Let $\{\vp_{k}\}_k\subset H_0^1(\Om)$ be the sequence of eigenfunctions of $(-\Delta, H_0^1(\Om))$ associated to the eigenvalues $\{\La_k\}_k$, then
$\La_k\to +\iy$ as $k\to+\iy$. Define maps $g^n=(g^n_1, g^n_2): A_n\to \R^{k^n_1-1}\times\R^{k_2-1}$ by
{\allowdisplaybreaks
\begin{align}\label{eq2-37}
g_1^n(\vec{u}):=\left(\int_{\Om}\vp_1u_1,\cdots, \int_{\Om}\vp_{k_1^{n}-2}u_1, \int_{\Om}|u_1|^3u_1\right),\\
g_2^n(\vec{u}):=\left(\int_{\Om}\vp_1u_2,\cdots, \int_{\Om}\vp_{k_2-2}u_2, \int_{\Om}|u_2|^3u_2\right).\nonumber
\end{align}
}%
Then $g^n\in F_{(k_1^n, k_2)}(A_n)$ and so there exists $\vec{u}^n=(u_1^n, u_2^n)\in A_n$ such that $g^n(\vec{u}^n)=0$.
As in the proof of Lemma \ref{lemma6}, firstly this means $\vec{u}^n\in A_n\setminus\mathcal{P}_{\dd_n}$ and so
$J_\bb (\vec{u}^n)\le C+1$ for every $n\in \mathbb{N}$. Secondly, we have $u_1^n\in\hbox{span}\{\vp_1,\cdots,\vp_{k_1^{n}-2}\}^{\bot}$ and so
$\int_{\Om}|\nabla u^n_1|^2\ge \La_{k_1^n-1}\int_{\Om}|u^n_1|^2$. Recall that
$$C+1\ge J_\bb (\vec{u}^n)\ge \frac{1}{4\mu_1}\|u_1^n\|^4_{\la_1},$$
we see that $u_1^n$ are uniformly bounded in $H_0^1(\Om)$. Up to a subsequence, we may assume that $u_1^n\to u_1$ weakly in $H_0^1(\Om)$
and strongly in $L^2(\Om)\cap L^4(\Om)$. Since
$$\int_{\Om}|u^n_1|^2\le \frac{1}{\La_{k_1^n-1}}\int_{\Om}|\nabla u^n_1|^2\to 0\quad\hbox{as}\,\,n\to+\iy,$$
so $u_1=0$. On the other hand, $\vec{u}^n\in A_n\subset\mathcal{M}$ yields $|u_1^n|_4=1$ for any $n$, so $|u_1|_4=1$, a contradiction.
Therefore (\ref{eq2-23}) holds and so (\ref{eq2}) has infinitely many sign-changing solutions $(u_{n,1}, u_{n, 2})$ such that
$E_\bb (u_{n,1}, u_{n, 2})\to +\iy$ as $n\to\iy$. By standard elliptic regularity theory, we see that $u_{n, i}\in L^{\iy}(\Om)$. Since
{\allowdisplaybreaks
\begin{align*}
4E_\bb (u_{n, 1}, u_{n, 2})&=\mu_1 |u_{n, 1}|_4^4+\mu_2|u_{n,2}|_4^4-2|\beta| \int_{\Om}u_{n,1}^2u_{n, 2}^2\\
&\le \mu_1 |\Om|\|u_{n, 1}\|^4_{L^\iy(\Om)}+\mu_2 |\Om|\|u_{n, 2}\|^4_{L^\iy(\Om)},
\end{align*}
}%
so (\ref{eq2-24}) holds. Here $|\Om|$ denotes the Lebesgue measure of $\Om$. This completes the proof.\end{proof}

\br\label{rmk4} If $A\in \Gamma^{(k_1, k_2)}\setminus \Gamma_\bb^{(k_1, k_2)}$,
we can not consider the set $\eta (2/C_0, A)$ in the proof of Theorem \ref{th1}, because $\eta (t, \cdot)$ can not be defined on the whole $\mathcal{M}$ for any $t>0$ and so $\eta (2/C_0, A)$ is not well defined.
Hence we can not replace $\Gamma_\bb^{(k_1, k_2)}$ by $\Gamma^{(k_1, k_2)}$ in the definition of $c_{\bb,\dd}^{k_1, k_2}$. Define
$$\widetilde{\Gamma}_{\bb}^{(k_1, k_2)}:=\left\{A\in \Gamma^{(k_1, k_2)} \,:\, \sup_{A}J_\bb< +\iy\right\}.$$
Then for any $A\in \widetilde{\Gamma}_{\bb}^{(k_1, k_2)}$, the set $B=\eta(2/C_0, A)$ is well defined.
Take $\vec{u}\in A$ such that $\eta(2/C_0, \vec{u})\in B\setminus \mathcal{P}_\dd$ as in the proof of Theorem \ref{th1}.
Then, since we do not know whether
$J_\bb(\vec{u})\le c^{k_1, k_2}+1$ holds or not, it seems impossible
for us to prove $\vec{u}\not\in \mathcal{P}_\dd$, which plays a crucial role in the proof of Theorem \ref{th1}.
Therefore we can not replace $\Gamma_\bb^{(k_1, k_2)}$ by $\widetilde{\Gamma}_{\bb}^{(k_1, k_2)}$ in the definition of $c_{\bb,\dd}^{k_1, k_2}$ either.
\er

\section{Proof of Theorem \ref{th2}}
\renewcommand{\theequation}{3.\arabic{equation}}

In this section,
let $k_1=k_2=2$ and take $\dd>0$ small enough such that $c_{\bb, \dd}^{2, 2}$ is a critical value of $E_\bb$.
Write $c_{\bb, \dd}^{2, 2}$ by $c$ for simplicity.
By the proof of Theorem \ref{th1} we see that (\ref{eq2}) has a sign-changing solution $\vec{U}=(U_1, U_2)$ such that
\be\label{eq3-1}E_\bb (\vec{U})=c\le c^{2, 2}.\ee
We will prove that $\vec{U}$ is a least energy sign-changing solution. To do this, let us define
\begin{align}\label{eq3-6}
\tilde{c}:=\inf_{\vec{u}\in\mathcal{N}_\bb} E_{\bb}(\vec{u}),
\end{align}
where
{\allowdisplaybreaks
\begin{align}\label{eq3-7}
\mathcal{N}_\bb:=\Big\{&\vec{u}=(u_1, u_2)\in H \,\,:\,\, \hbox{ both $u_1$ and $u_2$ change sign},\nonumber\\
 &E_\bb'(\vec{u})(u_1^{\pm}, 0)=0,\,\, E_\bb'(\vec{u})(0, u_2^{\pm})=0\Big\}.
\end{align}
}%
Then any sign-changing solutions belong to $\mathcal{N}_\bb$. In particular, $\vec{U}\in \mathcal{N}_\bb$ and so $\tilde{c}\le E_\bb (\vec{U})=c\le c^{2, 2}$.
To prove the opposite inequality $\tilde{c}\ge c$, we need the following lemma.

\bl\label{lemma12}Let $\vec{u}=(u_1, u_2)\in\mathcal{N}_\bb$, then
\be\label{eq3-1}E_\bb (u_1, u_2)=\sup_{t_1^{\pm}, t_2^{\pm}\ge 0}E_{\bb}\left(\sqrt{t_1^+}u_1^+-\sqrt{t_1^-}u_1^-, \,\sqrt{t_2^+}u_2^+-\sqrt{t_2^-}u_2^-\right).\ee\el

\begin{proof} Note that $E_\bb'(\vec{u})(u_1^{\pm}, 0)=0$ and $E_\bb'(\vec{u})(0, u_2^{\pm})=0$ yield
$$\mu_i|u_i^{\pm}|_4^4=\|u_i^{\pm}\|_{\la_i}^2+|\bb|\int_{\Om}|u_i^{\pm}|^2u_j^2,\quad i=1, 2.$$
Then
{\allowdisplaybreaks
\begin{align*}
&2|\bb|\int_{\Om}\left|\sqrt{t_1^+}u_1^+-\sqrt{t_1^-}u_1^-\right|^2\left|\sqrt{t_2^+}u_2^+-\sqrt{t_2^-}u_2^-\right|^2\\
=&2|\bb|t_1^+t_2^+\int_{\Om}(u_1^+)^2(u_2^+)^2+2|\bb|t_1^+t_2^-\int_{\Om}(u_1^+)^2(u_2^-)^2\\
&+2|\bb|t_1^-t_2^+\int_{\Om}(u_1^-)^2(u_2^+)^2+2|\bb|t_1^-t_2^-\int_{\Om}(u_1^-)^2(u_2^-)^2\\
\le& |\bb|\left[(t_1^+)^2+(t_2^+)^2\right]
\int_{\Om}(u_1^+)^2(u_2^+)^2+|\bb|\left[(t_1^+)^2+(t_2^-)^2\right]\int_{\Om}(u_1^+)^2(u_2^-)^2\\
&+|\bb|\left[(t_1^-)^2+(t_2^+)^2\right]
\int_{\Om}(u_1^-)^2(u_2^+)^2+|\bb|\left[(t_1^-)^2+(t_2^-)^2\right]\int_{\Om}(u_1^-)^2(u_2^-)^2\\
=&|\bb|(t_1^+)^2\int_{\Om}(u_1^+)^2u_2^2+|\bb|(t_1^-)^2\int_{\Om}(u_1^-)^2u_2^2\\
&+|\bb|(t_2^+)^2\int_{\Om}u_1^2(u_2^+)^2+|\bb|(t_2^-)^2\int_{\Om}u_1^2(u_2^-)^2\\
=&(t_1^+)^2\mu_1|u_1^+|_4^4+(t_1^-)^2\mu_1|u_1^-|_4^4+(t_2^+)^2\mu_2|u_2^+|_4^4+(t_2^-)^2\mu_2|u_2^-|_4^4\\
&-(t_1^+)^2\|u_1^+\|_{\la_1}^2-(t_1^-)^2\|u_1^-\|_{\la_1}^2-(t_2^+)^2\|u_2^+\|_{\la_2}^2-(t_2^-)^2\|u_2^-\|_{\la_2}^2.
\end{align*}
}%
Hence for any $t_1^{\pm}, t_2^{\pm}\ge 0$ we deduce that
{\allowdisplaybreaks
\begin{align*}
&E_{\bb}\left(\sqrt{t_1^+}u_1^+-\sqrt{t_1^-}u_1^-, \,\sqrt{t_2^+}u_2^+-\sqrt{t_2^-}u_2^-\right)\\
=&\frac{1}{2}t_1^+\|u_1^+\|_{\la_1}^2+\frac{1}{2}t_1^-\|u_1^-\|_{\la_1}^2+
\frac{1}{2}t_2^+\|u_2^+\|_{\la_2}^2+\frac{1}{2}t_2^-\|u_2^-\|_{\la_2}^2\\
&-\frac{1}{4}\Big[(t_1^+)^2\mu_1|u_1^+|_4^4+(t_1^-)^2\mu_1|u_1^-|_4^4
+(t_2^+)^2\mu_2|u_2^+|_4^4+(t_2^-)^2\mu_2|u_2^-|_4^4\Big]\\
&+\frac{1}{2}|\bb|\int_{\Om}\left|\sqrt{t_1^+}u_1^+-\sqrt{t_1^-}u_1^-\right|^2
\left|\sqrt{t_2^+}u_2^+-\sqrt{t_2^-}u_2^-\right|^2\\
\le&\sum_{i=1}^2\left(\frac{t_i^+}{2}-\frac{(t_i^+)^2}{4}\right)\|u_i^+\|_{\la_i}^2+
\sum_{i=1}^2\left(\frac{t_i^-}{2}-\frac{(t_i^-)^2}{4}\right)\|u_i^-\|_{\la_i}^2\\
\le &\frac{1}{4}\left(\|u_1^+\|_{\la_1}^2+\|u_1^-\|_{\la_1}^2+\|u_2^+\|_{\la_2}^2+\|u_2^-\|_{\la_2}^2\right)=E_\bb(u_1, u_2).
\end{align*}
}%
Letting $(t_1^+, t_1^-, t_2^+, t_2^-)=(1,1,1,1)$, we completes the proof.\end{proof}

\bl\label{lemma13} $c=\tilde{c}$ and so $\vec{U}$ is a least energy sign-changing solution of (\ref{eq2}).\el

\begin{proof} Take any $\vec{u}=(u_1, u_2)\in \mathcal{N}_\bb$ such that $E_\bb(\vec{u})<c^{2, 2}+1$. We define
$$A:=A_1\times A_2;\quad A_i:=\{u\in\hbox{span}\{u_i^+, u_i^-\} \,:\, |u|_4=1\}.$$
As in the proof of Lemma \ref{lemma7}, one has $A\in \Gamma^{(2, 2)}$. For any $\vec{v}=(v_1, v_2)\in A$, there exist
$b_i, d_i\in\R$ such that $v_i=b_i u_i^++d_i u_i^-$. Then by (\ref{eq2-9}) and Lemma \ref{lemma12} we have
{\allowdisplaybreaks
\begin{align*}
J_\bb (\vec{v})&=\sup_{t_1, t_2\ge 0}E_\bb (\sqrt{t_1}v_1, \,\sqrt{t_2} v_2)\\
&=\sup_{t_1, t_2\ge 0}E_\bb (\sqrt{t_1}(b_1 u_1^++d_1 u_1^-), \,\sqrt{t_2} (b_2 u_2^++d_2 u_2^-))\\
&=\sup_{t_1, t_2\ge 0}E_\bb \big(\sqrt{t_1}|b_1| u_1^+-\sqrt{t_1}|d_1| u_1^-, \,\sqrt{t_2} |b_2| u_2^+-\sqrt{t_2}|d_2| u_2^-\big)\\
&\le E_\bb(u_1, u_2),
\end{align*}
}%
that is, $\sup_A J_\bb\le E_\bb (\vec{u})<c^{2, 2}+1$ and so $A\in \Gamma^{(2, 2)}_\bb$, which implies
$$c=c_{\bb, \dd}^{2, 2}\le \sup_{\vec{v}\in
A\setminus\mathcal{P}_\dd}J_\bb(\vec{v})\le E_\bb (\vec{u}),\quad\forall\,\vec{u}\in \mathcal{N}_\bb\,\,\hbox{with}\,\,E_\bb(\vec{u})<c^{2, 2}+1.$$
Hence $c\le \tilde{c}$, that is, $\tilde{c}=c=E_\bb (\vec{U})$. Since any sign-changing solutions belong to $\mathcal{N}_\bb$, so $\vec{U}=(U_1, U_2)$ is a
least energy sign-changing solution of (\ref{eq2}).\end{proof}

To continue our proof, we need a classical result by Miranda.

\bl\label{lemma14} (see \cite{Miranda}) Consider a rectangle $\mathbf{R}=\prod_1^s [a_i, b_i]\subset \R^s$ and a continuous function $\Phi: \mathbf{R}\to \R^s$,
$\Phi= (\Phi_1, \cdots, \Phi_s)$. If $\Phi_i|_{x_i=a_i}>0>\Phi_i|_{x_i=b_i}$ holds for every $i$, then $\Phi$ has a zero inside $\mathbf{R}$.\el

\bl\label{lemma15}Both $U_1$ and $U_2$ has exactly two nodal domains.\el

\begin{proof} Since $U_1, U_2$ both change sign, so both $U_1$ and $U_2$ have at least two nodal domains.
Assume by contradiction that $U_1$ has at least three nodal domains $\Om_1, \Om_2$ and $\Om_3$. Without loss of generality, we assume that
$U_1>0$ on $\Om_1\cup\Om_2$.
Define $$u_1^+ := \chi_{\Omega_1} U_1, \quad u_1^- := \chi_{\Omega_2} U_1,\quad u_3 := \chi_{\Omega_3} U_1,$$ where
\begin{displaymath}
\chi_\Omega (x):=\begin{cases} 1, & x\in\Omega,\\
0, & x\in \RN\setminus \Omega.
\end{cases}\end{displaymath}
Then $u_1^{\pm}, u_3 \in H_0^1 (\Omega)\setminus\{0\}$. By $E_\bb' (\vec{U})(u_1^{\pm}, 0)=0$ and $E_\bb' (\vec{U})(0, U_2^{\pm})=0$ we have
{\allowdisplaybreaks
\begin{align}\label{eq3-2}
&\|u_1^{\pm}\|_{\la_1}^2=\mu_1|u_1^{\pm}|_4^4-|\bb|\int_{\Om}(u_1^{\pm})^2 U_2^2,\\
&\label{eq3-3}\|U_2^{\pm}\|_{\la_2}^2=\mu_2|U_2^{\pm}|_4^4-|\bb|\int_{\Om} U_1^2(U_2^{\pm})^2.
\end{align}
}%
Let
$$a:=\frac{1}{2}\min\left\{\frac{\|u_1^{\pm}\|_{\la_1}^2}
{\mu_1|u_1^{\pm}|_4^4},\quad\frac{\|U_2^{\pm}\|_{\la_2}^2}{\mu_2|U_2^{\pm}|_4^4}\right\}>0.$$
From (\ref{eq3-2})-(\ref{eq3-3}) one has $a<1/2$.
For any $b>1$, we define $\Phi=(f_1^+, f_1^-, f_2^+, f_2^-): [a, b]^4\to \R^4$ by
{\allowdisplaybreaks
\begin{align*}
f_1^{\pm}(t_1^+, t_1^-, t_2^+, t_2^-):=\|u_1^{\pm}\|_{\la_1}^2-t_1^{\pm}\mu_1|u_1^{\pm}|_4^4+|\bb|\int_{\Om}(u_1^{\pm})^2 \left|\sqrt{t_2^+}U_2^+-\sqrt{t_2^-}U_2^-\right|^2,\\
f_2^{\pm}(t_1^+, t_1^-, t_2^+, t_2^-):=\|U_2^{\pm}\|_{\la_2}^2-t_2^{\pm}\mu_2|U_2^{\pm}|_4^4
+|\bb|\int_{\Om}\left|\sqrt{t_1^+}u_1^+-\sqrt{t_1^-}u_1^-\right|^2(U_2^{\pm})^2.
\end{align*}
}%
Then for any $(t_1^+, t_1^-, t_2^+, t_2^-)\in [a, b]^4$,
{\allowdisplaybreaks
\begin{align*}
&f_1^{\pm}|_{t_1^{\pm}=a}\ge \|u_1^{\pm}\|_{\la_1}^2-a\mu_1|u_1^{\pm}|_4^4\ge \frac{1}{2}\|u_1^{\pm}\|_{\la_1}^2>0,\\
&f_2^{\pm}|_{t_2^{\pm}=a}\ge \|U_2^{\pm}\|_{\la_2}^2-a\mu_2|U_2^{\pm}|_4^4\ge \frac{1}{2}\|U_2^{\pm}\|_{\la_2}^2>0.
\end{align*}
}%
Moreover, by (\ref{eq3-2}) we have
{\allowdisplaybreaks
\begin{align}\label{eq3-4}
f_1^{\pm}|_{t_1^{\pm}=b}&=\|u_1^{\pm}\|_{\la_1}^2-b\mu_1|u_1^{\pm}|_4^4+|\bb|\int_{\Om}(u_1^{\pm})^2 \left|\sqrt{t_2^+}U_2^+-\sqrt{t_2^-}U_2^-\right|^2\nonumber\\
&=\|u_1^{\pm}\|_{\la_1}^2-b\mu_1|u_1^{\pm}|_4^4+|\bb|\int_{\Om}(u_1^{\pm})^2 \left(t_2^+(U_2^+)^2+t_2^-(U_2^-)^2\right)\nonumber\\
&\le \|u_1^{\pm}\|_{\la_1}^2-b\mu_1|u_1^{\pm}|_4^4+b|\bb|\int_{\Om}(u_1^{\pm})^2 U_2^2\nonumber\\
&=(1-b)\|u_1^{\pm}\|_{\la_1}^2<0,
\end{align}
}%
Similarly, by (\ref{eq3-3}) we have
\begin{align}\label{eq3-5}
f_2^{\pm}|_{t_2^\pm=b}\le\|U_2^{\pm}\|_{\la_2}^2-b\mu_2|U_2^{\pm}|_4^4+b|\bb|\int_{\Om}U_1^2(U_2^{\pm})^2
=(1-b)\|U_2^{\pm}\|_{\la_2}^2<0.
\end{align}
Then by Lemma \ref{lemma14} there exists $\left(\tilde{t}_1^+, \tilde{t}_1^-, \tilde{t}_2^+, \tilde{t}_2^-\right)\in [a, b]^4$ such that
$$f_1^{\pm}\left(\tilde{t}_1^+, \tilde{t}_1^-, \tilde{t}_2^+, \tilde{t}_2^-\right)=0,\quad f_2^{\pm}\left(\tilde{t}_1^+, \tilde{t}_1^-, \tilde{t}_2^+, \tilde{t}_2^-\right)=0.$$
This implies that
$$\left(\sqrt{\tilde{t}_1^+}u_1^+-\sqrt{\tilde{t}_1^-}u_1^-, \sqrt{\tilde{t}_2^+}U_2^+-\sqrt{\tilde{t}_2^-}U_2^-\right)\in \mathcal{N}_\bb.$$
Remark that (\ref{eq3-4}) and (\ref{eq3-5}) hold for any $b>1$, so we obtain that
$$\tilde{t}_1^+\le 1,\quad \tilde{t}_1^-\le 1,\quad \tilde{t}_2^+\le 1,\quad \tilde{t}_2^-\le 1.$$
Hence
{\allowdisplaybreaks
\begin{align*}
c&\le E_\bb\left(\sqrt{\tilde{t}_1^+}u_1^+-\sqrt{\tilde{t}_1^-}u_1^-, \sqrt{\tilde{t}_2^+}U_2^+-\sqrt{\tilde{t}_2^-}U_2^-\right)\\
&=\frac{1}{4}\left(\tilde{t}_1^+\|u_1^+\|_{\la_1}^2
+\tilde{t}_1^-\|u_1^-\|_{\la_1}^2+\tilde{t}_2^+\|U_2^+\|_{\la_2}^2+\tilde{t}_2^-\|U_2^-\|_{\la_2}^2\right)\\
&< \frac{1}{4}\left(\|u_1^+\|_{\la_1}^2+\|u_1^-\|_{\la_1}^2
+\|u_3\|_{\la_1}^2+\|U_2^+\|_{\la_2}^2+\|U_2^-\|_{\la_2}^2\right)\\
&\le \frac{1}{4}\left(\|U_1\|_{\la_1}^2+\|U_2\|_{\la_2}^2\right)=E_\bb(U_1, U_2)=c,
\end{align*}
}%
a contradiction. Hence $U_1$ has exactly two nodal domains. Similarly, $U_2$ has exactly two nodal domains.
\end{proof}

\begin{proof}[Proof of Theorem \ref{th2}.] Theorem \ref{th2} follows directly from Lemmas \ref{lemma13} and \ref{lemma15}.\end{proof}

\section{Proof of Theorem \ref{th3}}
\renewcommand{\theequation}{4.\arabic{equation}}

The following arguments are similar to those in Sections 2-3 with some important modifications. Here, although some definitions
are slight different from those in Section 2, we will use the same notations as in Section 2 for convenience.
To obtain semi-nodal solutions $(u_1, u_2)$ such that $u_1$ changes sign and $u_2$ is positive,
we consider the following functional
\begin{align*}
\widetilde{E}_\bb(u_1, u_2):=\frac{1}{2} \left(\|u_1\|_{\la_1}^2+\|u_2\|_{\la_2}^2\right)
-\frac{1}{4}\left(\mu_1 |u_1|_4^4+\mu_2|u_2^+|_4^4\right)+\frac{|\beta|}{2} \int_{\Om}u_1^2u_2^2,
\end{align*}
and modify the definition of $\wH$ by
$$\wH:=\{(u_1, u_2)\in H : u_1\neq 0,\,\,u_2^+\neq 0\}.$$
Then by similar proofs as in Section 2, we have the following lemmas.
\bl\label{lemma01}For any $(u_1, u_2)\in\wH$, if
$|\bb|^2(\int_{\Om}u_1^2 u_2^2)^2\ge \mu_1\mu_2|u_1|_4^4 |u_2^+|_4^4$,
then
$$\sup_{t_1, t_2\ge 0}\widetilde{E}_\bb (\sqrt{t_1}u_1, \,\sqrt{t_2} u_2)=+\iy.$$\el

\bl\label{lemma02}For any $\vec{u}=(u_1, u_2)\in\wH$, if
\be\label{eq02-3}|\bb|^2\left(\int_{\Om}u_1^2 u_2^2\right)^2< \mu_1\mu_2|u_1|_4^4 |u_2^+|_4^4,\ee
then system
\be\label{eq02-4}
\begin{cases}\|u_1\|_{\la_1}^2 =
t_1\mu_1 |u_1|_4^4-t_2|\beta|\int_{\Om} u_1^2 u_2^2\\
\|u_2\|_{\la_2}^2 =
t_2\mu_2 |u_2^+|_4^4-t_1|\beta|\int_{\Om} u_1^2 u_2^2\end{cases}\ee
has a unique solution
\be\label{eq02-5}
\begin{cases}t_1(\vec{u})=\frac{\mu_2|u_2^+|_4^4\|u_1\|_{\la_1}^2
+|\bb|\|u_2\|_{\la_2}^2\int_{\Om}u_1^2u_2^2}{\mu_1\mu_2|u_1|_4^4 |u_2^+|_4^4-|\bb|^2(\int_{\Om}u_1^2 u_2^2)^2}>0\\
t_2(\vec{u})=\frac{\mu_1|u_1|_4^4\|u_2\|_{\la_2}^2+|\bb|\|u_1\|_{\la_1}^2\int_{\Om}u_1^2u_2^2}{\mu_1\mu_2|u_1|_4^4 |u_2^+|_4^4-|\bb|^2(\int_{\Om}u_1^2 u_2^2)^2}>0.\end{cases}\ee
Moreover,
{\allowdisplaybreaks
\begin{align}\label{eq02-6}\sup_{t_1, t_2\ge 0}&\widetilde{E}_\bb \left(\sqrt{t_1}u_1, \,\sqrt{t_2} u_2\right)=\widetilde{E}_\bb\left(\sqrt{t_1(\vec{u})}u_1, \sqrt{t_2(\vec{u})}u_2\right)\nonumber\\
&=\frac{1}{4}\frac{\mu_2|u_2^+|_4^4\|u_1\|_{\la_1}^4
+2|\bb|\|u_1\|_{\la_1}^2\|u_2\|_{\la_2}^2\int_{\Om}u_1^2u_2^2+\mu_1|u_1|_4^4\|u_2\|_{\la_2}^4}{\mu_1\mu_2|u_1|_4^4 |u_2^+|_4^4-|\bb|^2(\int_{\Om}u_1^2 u_2^2)^2}
\end{align}
}%
and $(t_1(\vec{u}), t_2(\vec{u}))$ is the unique maximum point of $\widetilde{E}_\bb (\sqrt{t_1}u_1, \sqrt{t_2}u_2)$.
\el

Now, we modify the definitions of $\mathcal{M}^\ast$, $\mathcal{M}_\bb^\ast$, $\mathcal{M}_\bb^{\ast\ast}$, $\mathcal{M}$ and $\mathcal{M}_\bb$ by
{\allowdisplaybreaks
\begin{align}\label{eq02-8}
&\mathcal{M}^\ast:=\left\{\vec{u}\in H\,\,:\,\, |u_1|_4>1/2,\,\, |u_2^+|_4>1/2\right\};\nonumber\\
&\mathcal{M}_\bb^\ast:=\left\{\vec{u}\in \mathcal{M}^\ast\,\,:\,\, \hbox{$\vec{u}$ satisfies (\ref{eq02-3})}\right\};\nonumber\\
&\mathcal{M}_\bb^{\ast\ast}:=\left\{\vec{u}\in \mathcal{M}^\ast\,\,:\,\, \mu_1\mu_2-|\bb|^2\left(\int_{\Om}u_1^2 u_2^2\right)^2>0\right\};\nonumber\\
&\mathcal{M}:=\left\{\vec{u}\in H\,\,:\,\, |u_1|_4=1,\,\, |u_2^+|_4=1\right\},\quad \mathcal{M}_\bb:=\mathcal{M}\cap \mathcal{M}_\bb^\ast,
\end{align}
}%
and define a new functional $J_\bb : \mathcal{M}^\ast\to (0, +\iy]$ as in Section 2 by
$$J_\bb (\vec{u}):=\begin{cases}\frac{1}{4}\frac{\mu_2\|u_1\|_{\la_1}^4
+2|\bb|\|u_1\|_{\la_1}^2\|u_2\|_{\la_2}^2\int_{\Om}u_1^2u_2^2+\mu_1\|u_2\|_{\la_2}^4}{\mu_1\mu_2-|\bb|^2(\int_{\Om}u_1^2 u_2^2)^2}
&\hbox{if $\vec{u}\in \mathcal{M}^{\ast\ast}_\bb$},\\
+\iy &\hbox{if $\vec{u}\in \mathcal{M}^\ast\setminus\mathcal{M}^{\ast\ast}_\bb$}.\end{cases}$$
Then $J_\bb\in C(\mathcal{M}^\ast, (0, +\iy])$, $\inf_{\mathcal{M}^\ast}J_\bb\ge C_1>0$ where $C_1$ independent of $\bb<0$,
$J_\bb\in C^1(\mathcal{M}^{\ast\ast}_\bb, \,(0, +\iy))$
and (\ref{eq2-11})-(\ref{eq2-11-1})
hold for any $\vec{u}\in \mathcal{M}_\bb$ and $\vp,\, \psi\in H_0^1(\Om)$.
Note that
Lemmas \ref{lemma01} and \ref{lemma02} yield
\be\label{eq02-9}J_\bb (u_1, u_2)=\sup_{t_1, t_2\ge 0}\widetilde{E}_\bb \left(\sqrt{t_1}u_1, \,\sqrt{t_2} u_2\right),\quad \forall\,(u_1, u_2)\in \mathcal{M}.\ee

For any $\vec{u}=(u_1, u_2)\in \mathcal{M}_\bb^*$, let $\tilde{w}_i\in H_0^1(\Om)$, $i=1, 2$, be the unique
solutions of the following linear problem
\be\label{eq02-012}\begin{cases}
-\Delta \tilde{w}_1+\la_1 \tilde{w}_1+|\bb|t_2(\vec{u})u_2^2 \tilde{w}_1=\mu_1 t_1(\vec{u})u_1^3,\quad \tilde{w}_1\in H_0^1(\Om),\\
-\Delta \tilde{w}_2+\la_2 \tilde{w}_2+|\bb|t_1(\vec{u})u_1^2 \tilde{w}_2=\mu_2 t_2(\vec{u})(u_2^+)^3,\quad \tilde{w}_2\in H_0^1(\Om).
\end{cases}\ee
As in Section 2, we define
\be\label{eq02-12}w_i=\al_i \tilde{w}_i,\quad\hbox{where}\,\,\al_1=\frac{1}{\int_{\Om}u_1^3 \tilde{w}_1}>0,\,\,
\al_2=\frac{1}{\int_{\Om}(u_2^+)^3 \tilde{w}_2}>0 .\ee
Then $(w_1, w_2)$ is the unique solution of the problem
\be\label{eq02-13}\begin{cases}
-\Delta w_1+\la_1 w_1+|\bb|t_2(\vec{u})u_2^2 w_1=\al_1\mu_1 t_1(\vec{u})u_1^3,\quad w_1\in H_0^1(\Om),\\
-\Delta w_2+\la_2 w_2+|\bb|t_1(\vec{u})u_1^2 w_2=\al_2\mu_2 t_2(\vec{u})(u_2^+)^3,\quad w_2\in H_0^1(\Om),\\
\int_{\Om}u_1^3 w_1\,dx=1,\quad \int_{\Om}(u_2^+)^3 w_2\,dx=1.
\end{cases}\ee
As in Section 2, the operator $K=(K_1, K_2) : \mathcal{M}_\bb^\ast\to H$ is defined as
$K(\vec{u}):=\vec{w}=(w_1, w_2)$, and similar arguments as Lemma \ref{lemma3}
yield $K\in C^1(\mathcal{M}_\bb^*, H)$. Since
$u_n\to u$ in $L^4(\Om)$ implies $u_n^+\to u^+$ in $L^4(\Om)$,
so Lemma \ref{lemma4} and its proof with obvious modifications also hold for this new $K$ defined here. Note that
\be\label{eq02-15}K(\sg_1(\vec{u}))=\sg_1(K(\vec{u})).\ee
 Remark that (\ref{eq02-15}) only holds for $\sg_1$ and in the sequel we only use $\sg_1$.
Consider
$$\mathcal{F}=\{A\subset \mathcal{M} : A\,\,\hbox{ is closed and}\,\,\sg_1(\vec{u})\in A\,\,\,\forall\,\,\vec{u}\in A\},$$
and, for each $A\in \mathcal{F}$ and $k_1\ge 2$, the class of functions
$$F_{(k_1, 1)}(A)=\left\{f: A\to \R^{k_1-1} \,:  \,f\,\,\hbox{continuous and}\,\,
 f(\sg_1(\vec{u}))=-f(\vec{u})
 \right\}.$$
\begin{definition} (Modified vector genus, slightly different from Definition \ref{definition1})
Let $A\in \mathcal{F}$ and take any $k_1\in\mathbb{N}$ with $k_1\ge 2$. We say that $\vec{\ga}(A)\ge (k_1, 1)$ if for every $f\in F_{(k_1, 1)}(A)$ there exists $\vec{u}\in A$ such that $f(\vec{u})=0$. We denote
$$\Gamma^{(k_1, 1)}:=\{A\in\mathcal{F} : \vec{\ga}(A)\ge (k_1, 1)\}.$$
\end{definition}

\bl\label{lemma05} With the previous notations, the following properties hold.
\begin{itemize}
\item[$(i)$] Take $A:=A_1\times A_2\subset \mathcal{M}$ and let $\eta: S^{k_1-1}\to A_1$ be a homeomorphism such that
$\eta(-x)=-\eta(x)$ for every $x\in S^{k_1-1}$. Then $A\in \Gamma^{(k_1, 1)}$.

\item[$(ii)$] We have $\overline{\eta(A)}\in \Gamma^{(k_1, 1)}$ whenever $A\in \Gamma^{(k_1, 1)}$ and a continuous map $\eta : A\to \mathcal{M}$ is such
that $\eta\circ \sg_1=\sg_1\circ\eta$.
\end{itemize}\el

\begin{proof} The conclusion (ii) is trivial, we only prove (i).
Fix any $f\in F_{(k_1, 1)}(A)$ and take any $u_2\in A_2$. Define
$\vp: S^{k_1-1}\to \R^{k_1-1}$ by $\vp(x):=f(\eta(x), u_2)$.
Then $\vp$ is continuous and $\vp(-x)=-\vp(x)$. So by Borsuk-Ulam Theorem,
there exists $x_0\in S^{k_1-1}$ such that $\vp(x_0)=0$, that is $f(\eta(x_0), u_2)=0$. Hence $\vec{\ga}(A)\ge (k_1, 1)$ and $A\in\Gamma^{(k_1, 1)}$.\end{proof}

Now we modify the definitions of
$\mathcal{P}$ and $\hbox{dist}_4(\vec{u}, \mathcal{P})$ in (\ref{cone})-(\ref{cone1}) by
\be\label{eq4-1}\mathcal{P}:=\mathcal{P}_1\cup -\mathcal{P}_1,\quad\hbox{dist}_4(\vec{u}, \mathcal{P}):=\min\big\{\hbox{dist}_4(u_1,\,\mathcal{P}_1),\,\,\hbox{dist}_4(u_1,\,-\mathcal{P}_1)\big\}.\ee
Under this new definition, $u_1$ changes sign if $\hbox{dist}_4(\vec{u}, \mathcal{P})>0$.

\bl\label{lemma16} Let $k_1\ge 2$. Then for any $\dd<2^{-1/4}$ and any $A\in\Gamma^{(k_1, 1)}$ there holds $A\setminus \mathcal{P}_\dd\neq\emptyset$.\el

\begin{proof} Fix any $A\in\Gamma^{(k_1, 1)}$. Recall the map
$f_1 : A\to \R^{k_1-1}$ defined in (\ref{eq2-36}).
Clearly $f_1\in F_{(k_1, 1)}(A)$, so there exists $\vec{u}\in A$ such that $f_1(\vec{u})=0$, which means
$\int_{\Om}(u_1^{+})^4=\int_{\Om}(u_1^-)^4=1/2$,
that is, $\hbox{dist}_4(\vec{u}, \mathcal{P})=2^{-1/4}$, so $\vec{u}\in A\setminus \mathcal{P}_\dd$ for every $\dd<2^{-1/4}$.\end{proof}

\bl\label{lemma07} Let $k_1\ge 2$. There exist $A\in\Gamma^{(k_1, 1)}$ and a constant $c^{k_1, 1}\in\mathbb{N}$ independent of $\bb<0$
such that $\sup_{A}J_\bb\le c^{k_1, 1}$ for any $\bb<0$.\el

\begin{proof} Let $B_i$ and $\{\vp^i_{k} : 1\le k\le k_i\}\subset H_0^1(B_i)$ be in the proof of Lemma \ref{lemma7}. Define
$$A_1:=\big\{u\in \hbox{span}\{\vp^1_{1},\cdots, \vp_{k_1}^1\} \,:\, |u|_4=1\big\}, \quad A_2=\left\{C\left|\vp_1^2\right| : C=1/|\vp_1^2|_4\right\}.$$
Then by Lemma \ref{lemma05}-$(i)$ one has $A:=A_1\times A_2\in \Gamma^{(k_1, 1)}$. The rest of the proof is the same as Lemma \ref{lemma7}.
\end{proof}

For every $k_1\ge 2$ and $0<\dd < 2^{-1/4}$, we define
$$ c_{\bb,\dd}^{k_1, 1}:=\inf_{A\in \Gamma_\bb^{(k_1, 1)}}\sup_{\vec{u}\in A\setminus \mathcal{P}_\dd}J_\bb(\vec{u}),$$
where the definition of $\Gamma_\bb^{(k_1, 1)}$ is the same as (\ref{eq2-17-1}).
Then Lemma \ref{lemma07} yields $\Gamma_\bb^{(k_1, 1)}\neq \emptyset$ and so $c_{\bb, \dd}^{k_1, 1}$ is well defined. Moreover,
$c_{\bb, \dd}^{k_1, 1}\le c^{k_1, 1}$ for any $\bb<0$ and $\dd>0.$
Under the new definitions (\ref{eq4-1}), it is easy to see that Lemma \ref{lemma8} also holds here.
Now as in Section 2, we define a map
$V : \mathcal{M}_\bb^\ast \to H$ by $V(\vec{u}):=\vec{u}-K(\vec{u}).$
Then Lemma \ref{lemma9} also holds here. Recall from (\ref{eq02-8}) and (\ref{eq02-13}) that
$$\int_{\Om}(u_2^+)^3 (u_2-w_2)\,dx=1-1=0,\quad\forall\,\, \vec{u}=(u_1, u_2)\in \mathcal{M}_\bb.$$
Then by similar arguments, we see that Lemma \ref{lemma10} also holds here.

\begin{lemma}\label{lemma011} There exists a unique global solution $\eta=(\eta_1, \eta_2) : [0, \iy)\times \mathcal{M}_\bb \to H$ for the initial value problem
\be\label{eq02-19}\frac{d}{dt}\eta(t,\vec{u})=-V(\eta(t, \vec{u})), \quad \eta(0, \vec{u})=\vec{u}\in\mathcal{M}_\bb.\ee
Moreover, conclusions $(i)$, $(iii)$ and $(iv)$ of Lemma \ref{lemma11} also hold here, and
$\eta(t, \sg_1(\vec{u}))=\sg_1(\eta(t, \vec{u}))$ for any $t>0$ and $u\in\mathcal{M}_\bb$.
\end{lemma}

\begin{proof} Recalling $V(\vec{u})\in C^1(\mathcal{M}_\bb^\ast, H)$,
 (\ref{eq02-19}) has a unique solution
$\eta: [0, T_{\max})\times\mathcal{M}_\bb\to H$,
where $T_{\max}>0$ is the maximal time such that $\eta(t, \vec{u})\in \mathcal{M}_\bb^\ast$ for all
$t\in [0, T_{\max})$.
Fix any $\vec{u}=(u_1, u_2)\in \mathcal{M}_\bb$, we deduce from (\ref{eq02-19}) that
{\allowdisplaybreaks
\begin{align*}
\frac{d}{dt}\int_{\Om}\left(\eta_2(t, \vec{u})^+\right)^4\,dx &=-4\int_{\Om}\left(\eta_2(t, \vec{u})^+\right)^3(\eta_2(t, \vec{u})-K_2(\eta(t, \vec{u})))\,dx\\
&=4-4\int_{\Om}\left(\eta_2(t, \vec{u})^+\right)^4\,dx,\quad \forall\, 0< t<T_{\max}.
\end{align*}
}%
that is
$$\frac{d}{dt}\left[e^{4t}\left(\int_{\Om}\left(\eta_2(t, \vec{u})^+\right)^4\,dx-1\right)\right]=0.$$
Since $\int_{\Om}\left(\eta_2(0, \vec{u})^+\right)^4dx=\int_{\Om}(u_2^+)^4dx=1$, so
$\int_{\Om}\left(\eta_2(t, \vec{u})^+\right)^4 dx\equiv 1$
for all $0\le t<T_{\max}$.
Recalling (\ref{eq02-15}), the rest of the proof is similar to Lemma \ref{lemma11}.\end{proof}

Now we can give the proof of Theorem \ref{th3}.

\begin{proof}[Proof of Theorem \ref{th3}.]
{\bf Step 1.} Fix any $k_1\ge 2$. We prove that $c_{\bb, \dd}^{k_1, 1}$ is a sign-changing critical value of $E_\bb$ for $\dd>0$ small.

 By similar arguments as Step 1 in the proof of Theorem \ref{th1}, for small $\dd>0$,
there exists $\vec{u}=(u_1, u_2)\in\mathcal{M}_\bb$ such that
$$J_\bb(\vec{u})=c_{\bb,\dd}^{k_1, 1},\quad V(\vec{u})=0\quad\hbox{and dist}_4(\vec{u}, \mathcal{P})\ge \dd.$$
Then $u_1$ changes sign. Since $V(\vec{u})=0$, so $\vec{u}=K(\vec{u})$. Combining this with (\ref{eq02-13}), we see that $\vec{u}$ satisfies
\be\label{eq02-22-1}\begin{cases}
-\Delta u_1+\la_1 u_1+|\bb|t_2(\vec{u})u_2^2 u_1=\al_1\mu_1 t_1(\vec{u})u_1^3,\\
-\Delta u_2+\la_2 u_2+|\bb|t_1(\vec{u})u_1^2 u_2=\al_2\mu_2 t_2(\vec{u})(u_2^+)^3.
\end{cases}\ee
Since $|u_1|_4=1$, $|u_2^+|_4=1$ and $t_i(\vec{u})$ satisfies (\ref{eq02-4}), so $\al_1=\al_2=1$.
Multiply the second equation of (\ref{eq02-22-1}) by $u_2^-$ and integrate over $\Om$, we get $\|u_2^-\|_{\la_2}^2=0$, so $u_2\ge 0$. By the strong
maximum principle, $u_2>0$ in $\Om$. Hence $(\tilde{u}_1, \tilde{u}_2):=(\sqrt{t_1(\vec{u})}u_1, \sqrt{t_2(\vec{u})}u_2)$ is a semi-nodal solution of the original problem (\ref{eq2}) with $\tilde{u}_1$ sign-changing and $\tilde{u}_2$ positive. Moreover, (\ref{eq02-6}) and (\ref{eq02-9}) yield
$$E_\bb (\tilde{u}_1, \tilde{u}_2)=\widetilde{E}_\bb (\tilde{u}_1, \tilde{u}_2)=J_\bb (u_1, u_2)=c_{\bb, \dd}^{k_1, 1}\le c^{k_1, 1}.$$

{\bf Step 2.} We prove that (\ref{eq2}) has infinitely many semi-nodal solutions.

Assume by contradiction that there exist $k_1^n\to\iy$, $\dd_n\in (0, 2^{-5/4}]$ and a positive constant $C$ such that
$c_{\bb, \dd_n}^{k_1^n, 1}\le C$ for every $n\in\mathbb{N}$.
Then there exists $A_n\in \Gamma_\bb^{(k_1^n, 1)}$ such that
$\sup_{A_n\setminus\mathcal{P}_{\dd_n}} J_\bb\le C+1$ for any $n\in\mathbb{N}$.
Let $\{\vp_{k}\}_k\subset H_0^1(\Om)$ be in the proof of Theorem \ref{th1} and recall the map $g^n_1: A_n\to \R^{k^n_1-1}$ defined in (\ref{eq2-37}).
Then $g_1^n\in F_{(k_1^n, 1)}(A_n)$. By the same arguments as in the proof of Theorem \ref{th1}, we get a contradiction. Therefore, (\ref{eq2}) has infinitely many semi-nodal solutions $\{\vec{u}_n=(u_{n, 1}, u_{n, 2})\}_{n\ge 2}$ which satisfy
\begin{itemize}
\item[$(1)$] $u_{n, 1}$ changes sign and $u_{n, 2}$ is positive;
\item[$(2)$] $E_\bb(u_{n, 1}, u_{n, 2})=c_{\bb, \dd_n}^{n, 1}\le c^{n, 1}$ for some $0<\dd_n<2^{-1/4}$. Moreover,
$$\|u_{n, 1}\|_{L^\iy(\Om)}+\|u_{n,2}\|_{L^\iy(\Om)}\to+\iy\quad\hbox{as}\,\,\,n\to+\iy.$$
\end{itemize}

{\bf Step 3.} We prove that $u_{n, 1}$ has at most $n$ nodal domains.

Assume that $u_{n, 1}$ has at least $n+1$ nodal domains $\Om_k, 1\le k\le n+1$, then $u_{n,1}\chi_{\Om_k}\in H_0^1(\Om)$.
For $1\le k\le n$, we see from $E_\bb'(\vec{u}_n)(u_{n,1}\chi_{\Om_k}, 0)=0$ that
\be\label{eq4-4}\mu_1 |u_{n,1}\chi_{\Om_k}|_4^4=\|u_{n,1}\chi_{\Om_k}\|_{\la_1}^2+|\bb|\int_{\Om} |u_{n,1}\chi_{\Om_k}|^2 u_{n,2}^2,\quad k=1,\cdots, n.\ee
Similarly $E_\bb'(\vec{u}_n)(0, u_{n,2})=0$ yields
$$\mu_2|u_{n,2}|_4^4=\|u_{n,2}\|_{\la_2}^2+|\bb|\int_{\Om}u_{n,1}^2u_{n,2}^2.$$
Then, similarly as Lemma \ref{lemma12} we have
{\allowdisplaybreaks
\begin{align*}
&2|\bb|\int_{\Om}\left|\sum\nolimits_{k=1}^n \sqrt{t_k}u_{n,1}\chi_{\Om_k}\right|^2 |\sqrt{s}u_{n,2}|^2\\
\le&\sum_{k=1}^n |\bb|t_k^2\int_{\Om}|u_{n,1}\chi_{\Om_k}|^2 u_{n,2}^2+|\bb|s^2\int_{\Om}u_{n,1}^2u_{n,2}^2\\
\le&\sum_{k=1}^n t_k^2\left(\mu_1 |u_{n,1}\chi_{\Om_k}|_4^4-\|u_{n,1}\chi_{\Om_k}\|_{\la_1}^2\right)+s^2\left(\mu_{n,2}|u_{n,2}|_4^4-\|u_{n,2}\|_{\la_2}^2\right).
\end{align*}
}%
Recall that $u_{n, 2}$ is positive, so for $t_1,\cdots, t_n, s\ge 0$,
{\allowdisplaybreaks
\begin{align}\label{eq4-3}
&\widetilde{E}_\bb\left(\sum_{k=1}^n \sqrt{t_k}u_{n,1}\chi_{\Om_k}, \sqrt{s}u_{n,2}\right)\nonumber\\
=&\sum_{k=1}^{n}\left(\frac{t_k}{2}\|u_{n,1}\chi_{\Om_k}\|_{\la_1}^2-\frac{t_k^2}{4}\mu_1|u_{n,1}\chi_{\Om_k}|_4^4\right)
+\left(\frac{s}{2}\|u_{n,2}\|_{\la_2}^2-\frac{s^2}{4}\mu_2|u_{n,2}|_4^4\right)\nonumber\\
&+\frac{|\bb|}{2}\int_{\Om}\left|\sum\nolimits_{k=1}^n \sqrt{t_k}u_{n,1}\chi_{\Om_k}\right|^2 |\sqrt{s}u_{n,2}|^2\nonumber\\
=&\sum_{k=1}^{n}\left(\frac{t_k}{2}-\frac{t_k^2}{4}\right)\|u_{n,1}\chi_{\Om_k}\|_{\la_1}^2+\left(\frac{s}{2}
-\frac{s^2}{4}\right)\|u_{n,2}\|_{\la_2}^2\nonumber\\
\le &\frac{1}{4}\sum_{k=1}^{n}\|u_{n,1}\chi_{\Om_k}\|_{\la_1}^2+\frac{1}{4}\|u_{n,2}\|_{\la_2}^2.
\end{align}
}%
Now we define
{\allowdisplaybreaks
\begin{align*}&A:=A_1\times \left\{C u_{n, 2} : C=1/|u_{n,2}|_4\right\};\\
 &A_1:=\big\{u\in\hbox{span}\{u_{n,1}\chi_{\Om_1},\cdots,u_{n,1}\chi_{\Om_n}\} : |u|_4=1\big\}.\end{align*}
}%
Then Lemma \ref{lemma05}-$(i)$ yields $A\in \Gamma^{(n, 1)}$, and similarly as Lemma \ref{lemma13},
we deduce from (\ref{eq02-9}) and (\ref{eq4-3}) that
{\allowdisplaybreaks
\begin{align*}
\sup_A J_\bb&\le \frac{1}{4}\sum_{k=1}^{n}\|u_{n,1}\chi_{\Om_k}\|_{\la_1}^2+\frac{1}{4}\|u_{n,2}\|_{\la_2}^2<\frac{1}{4}\left(\|u_{n, 1}\|_{\la_1}^2+\|u_{n, 2}\|_{\la_2}^2\right)\\
&=E_\bb(\vec{u}_n)=c_{\bb, \dd_n}^{n, 1}\le c^{n, 1},
\end{align*}
}%
and so $A\in \Gamma^{(n, 1)}_\bb$, which implies
\be\label{eq4-5} c_{\bb, \dd_n}^{n, 1}\le \sup_{A\setminus\mathcal{P}_{\dd_n}}J_\bb\le \frac{1}{4}\sum_{k=1}^{n}\|u_{n,1}\chi_{\Om_k}\|_{\la_1}^2+\frac{1}{4}\|u_{n,2}\|_{\la_2}^2<E_\bb (\vec{u}_n),\ee
a contradiction. Hence $u_{n, 1}$ has at most $n$ nodal domains. In particular, $u_{2, 1}$ has exactly two nodal domains.

{\bf Step 4.} We prove that $(u_{2, 1}, u_{2, 2})$ has the least energy among all nontrivial solutions whose first component changes sign.

By similar arguments as in Section 3, we can prove that
\begin{align}\label{eq4-2}
c_{\bb, \dd_2}^{2, 1}=\inf_{\vec{u}\in\mathcal{N}_{2,1, \bb}} E_{\bb}(\vec{u})=\inf_{\vec{u}\in\mathcal{N}_{2,1, \bb}} \widetilde{E}_{\bb}(\vec{u}),
\end{align}
where
{\allowdisplaybreaks
\begin{align*}
\mathcal{N}_{2,1, \bb}:=\Big\{&\vec{u}=(u_1, u_2)\in H \,\,:\,\, \hbox{$u_1$ changes sign and $u_2\ge 0,\,\,u_2\neq 0$},\\
 &\quad E_\bb'(\vec{u})(u_1^{\pm}, 0)=0,\,\, E_\bb'(\vec{u})(0, u_2)=0\Big\}.
\end{align*}
}%
Let $\vec{u}=(u_1, u_2)$ be any a nontrivial solution of (\ref{eq2}) with $u_1$ sign-changing. Without loss of generality we assume $u_2^+\neq 0$.
Then by a similar argument as Lemma \ref{lemma15}, there exists $t_1^\pm, t_2^+\in (0, 1]$ such that
$$\left(\sqrt{t_1^+}u_1^+-\sqrt{t_1^-}u_1^-, \sqrt{t_2^+}u_2^+\right)\in \mathcal{N}_{2, 1, \bb},$$
and so
{\allowdisplaybreaks
\begin{align*}
E_\bb(u_{2, 1}, u_{2,2})&=c_{\bb, \dd_2}^{2, 1}\le E\left(\sqrt{t_1^+}u_1^+-\sqrt{t_1^-}u_1^-, \sqrt{t_2^+}u_2^+\right)\\
&\le\frac{1}{4}(\|u_1\|_{\la_1}^2+\|u_2^+\|_{\la_2}^2)\le E_\bb(\vec{u}).\end{align*}
}%
Hence $(u_{2, 1}, u_{2, 2})$ has the least energy among all nontrivial solutions whose first component changes sign.
This completes the proof. \end{proof}

\section{Asymptotic behaviors and  phase seperation}
\renewcommand{\theequation}{5.\arabic{equation}}

In this section, we study the limit behavior of solutions obtained above as $\bb\to-\iy$. Fix any $k_1, k_2\in\mathbb{N}$ such that
$k_1\ge 2$ and $k_2\ge 1$.
Then by the arguments in above sections we know that, for any $\bb<0$, there exists $\dd_{\bb}\in (0, 2^{-1/4})$
and $\vec{u}_\bb=(u_{\bb, 1}, u_{\bb, 2})\in H$ such that $\vec{u}_\bb$ is a nontrivial solution (either sign-changing or semi-nodal) of (\ref{eq2}) with
$$E_\bb(\vec{u}_\bb)=c_{\bb, \dd_{\bb}}^{k_1, k_2}\le c^{k_1, k_2}<+\iy.$$
Here $c^{k_1, k_2}$ is seen in Lemmas \ref{lemma7} and \ref{lemma07}. Recall that
$$E_\bb(\vec{u}_\bb)=\frac{1}{4}\left(\|u_{\bb, 1}\|_{\la_1}^2+\|u_{\bb, 2}\|_{\la_2}^2\right),$$
we see that $\vec{u}_\bb$ are uniformly bounded in $H$. On the other hand, by Kato's inequality (see \cite{KA}) we have
$$\Delta |u_{\bb, i}|\ge \Delta u_{\bb, i} \cdot\frac{u_{\bb, i}}{|u_{\bb,i}|}
\quad
\hbox{in}\,\, (H_0^1(\Om))'.$$ Recall that $\bb<0$, then it is easy to check that
$$-\Delta |u_{\bb, i}|+\la_i |u_{\bb, i}|\le\mu_i |u_{\bb, i}|^3,\quad |u_{\bb, i}|\in H_0^1(\Om).$$
Hence by standard Moser iteration, we see that $u_{\bb, i}$ are uniformly bounded in $L^\iy(\Om)$ for any $\bb<0$ and $i=1, 2$.
Moreover, by elliptic regularity theory it holds that $u_{\bb, i}\in C(\overline{\Om})\cap C^2(\Om)$. The main result of this section is following.

\bt\label{th4} There exists a vector Lipschitz function $\vec{u}_\iy=(u_{\iy, 1}, u_{\iy, 2})\in \widetilde{H}$ such that, up to a subsequence,
\begin{itemize}
\item[$(1)$] $u_{\bb, i}\to u_{\iy, i} $ in $H_0^1(\Om)\cap C^{0, \al}(\overline{\Om})$ for every $0<\al<1$ as $\bb\to-\iy$;
\item[$(2)$] $-\Delta u_{\iy, i}+\la_i u_{\iy, i}=\mu_i u_{\iy, i}^3$ in the open set $\{u_{\iy, i}\neq 0\}$;
\item[$(3)$] $u_{\iy, 1}\cdot u_{\iy, 2}\equiv 0$ and $|\bb|\int_{\Om} u_{\bb, 1}^2 u_{\bb, 2}^2\,dx\to 0$ as $\bb\to-\iy$;
\item[$(4)$] if $k_1, k_2\ge 2$, then $u_{\iy, i}$ changes sign for $i=1, 2$. Moreover, if $k_1=k_2=2$, then $\{u_{\iy, i}\neq 0\}$ has exactly two connected components, and $u_{\iy, i}$ is a least energy
sign-changing solution of
\be\label{eq5-1}-\Delta u+\la_i u=\mu_i u^3,\quad u\in H_0^1(\{u_{\iy, i}\neq 0\})\ee for $i=1, 2$;
\item[$(5)$] if $k_1\ge 2$ and $k_2=1$, then $u_{\iy, 1}$ changes sign, $\{u_{\iy, 1}\neq 0\}$ has at most $k_1$ connected components and $u_{\iy, 2}$ is positive in $\{u_{\iy, 2}\neq 0\}$. Moreover,
 if $(k_1, k_2)=(2, 1)$, then $\{u_{\iy, 1}\neq 0\}$ has exactly two connected components,
 $u_{\iy, 1}$ is a least energy
sign-changing solution of (\ref{eq5-1}) for $i=1$, $\{u_{\iy, 2}\neq 0\}$ is connected, and $u_{\iy, 2}$ is a least energy solution of (\ref{eq5-1}) for $i=2$.
\end{itemize}
\et

\begin{proof} Since $u_{\bb, i}$ are uniformly bounded in $L^\iy(\Om)$ for any $\bb<0$ and $i=1, 2$, then $(1)-(3)$ follows from
\cite[Theorems 1.1 and 1.2]{NTTV}. Remark that, although in \cite{NTTV} the
results are stated for nonnegative solutions, they also hold for solutions with no sign-restrictions; all arguments there can be adapted with little extra
effort to this more general case, working with the positive and negative parts of a solution. This fact was pointed out in the proof of \cite[Theorem 4.3]{TT}.

It remains to prove $(4)-(5)$. First we consider the case $k_1, k_2\ge 2$. Since $u_{\bb, i}\in C(\overline{\Om})\cap C^2(\Om)$ and $u_{\bb, i}$ changes sign, so there exists $x_{\bb,i}^\pm\in \Om$ such that
$$u_{\bb, i}(x_{\bb, i}^+)=\max_{x\in\Om} u_{\bb, i}(x)>0\quad\hbox{and}\quad u_{\bb, i}(x_{\bb, i}^-)=\min_{x\in\Om} u_{\bb, i}(x)<0,\quad i=1, 2.$$
Then $\Delta u_{\bb, i}(x_{\bb, i}^+)\le 0$. Since $\vec{u}_\bb$ satisfies (\ref{eq2}) and $\bb<0$, so
$\la_i u_{\bb, i}(x_{\bb, i}^+)\le\mu_i u_{\bb, i}^3(x_{\bb, i}^+)$,
which implies
$$u_{\bb, i}(x_{\bb, i}^+)=\max_{x\in\Om} u_{\bb, i}(x)\ge\sqrt{\la_i/\mu_i},\quad\forall\,\bb<0.$$
Similarly,
$$u_{\bb, i}(x_{\bb, i}^-)=\min_{x\in\Om} u_{\bb, i}(x)\le-\sqrt{\la_i/\mu_i},\quad\forall\,\bb<0.$$
Combining these with $(1)$, we see that $u_{\iy, i}$ changes sign, and so $\{u_{\iy, i}\neq 0\}$ has at least two connected components.

Now we let $(k_1, k_2)=(2, 2)$. Assume by contradiction that $\{u_{\iy, 1}\neq 0\}$ has at least three connected components
$\Om_1, \Om_2$ and $\Om_3$. Without loss of generality, we assume that
$u_{\iy, 1}>0$ on $\Om_1\cup\Om_2$. As in the proof of Lemma \ref{lemma15}, we
define $u_1^+ := \chi_{\Omega_1} u_{\iy, 1}$, $u_1^- := \chi_{\Omega_2} u_{\iy, 1}$ and $u_3 := \chi_{\Omega_3} u_{\iy, 1}$,
then $\|u_{\iy, 1}\|_{\la_1}^2>\|u_1^+\|_{\la_1}^2+\|u_1^-\|_{\la_1}^2$. By Theorem \ref{th4}-$(2)$ and (\ref{eq3-7}) it is easy to see that $(u_1^+-u_1^-, u_{\iy, 2})\in \mathcal{N}_\bb$ for all $\bb<0$, so
{\allowdisplaybreaks
\begin{align}\label{eq5-2}
\frac{1}{4}\left(\|u_{\iy, 1}\|_{\la_1}^2+\|u_{\iy, 2}\|_{\la_2}^2\right)&=\lim_{\bb\to-\iy}\frac{1}{4}\left(\|u_{\bb, 1}\|_{\la_1}^2+\|u_{\bb, 2}\|_{\la_2}^2\right)\nonumber\\
&=\lim_{\bb\to-\iy}E_\bb(u_{\bb, 1}, u_{\bb, 2})=\lim_{\bb\to-\iy}c_{\bb, \dd_\bb}^{2, 2}\nonumber\\
&\le\lim_{\bb\to-\iy}E_\bb(u_1^+-u_1^-, u_{\iy, 2})\nonumber\\
&=\frac{1}{4}\left(\|u_1^+\|_{\la_1}^2+\|u_1^-\|_{\la_1}^2+\|u_{\iy, 2}\|_{\la_2}^2\right),
\end{align}
}%
a contradiction. Hence $\{u_{\iy, i}\neq 0\}$ has exactly two connected components. If $u_i\in H_0^1(\{u_{\iy, i}\neq 0\})$ is
any sign-changing solutions of (\ref{eq5-1}),
then $(u_1, u_{\iy, 2})$, $(u_{\iy, 1}, u_2)\in \mathcal{N}_\bb$ for all $\bb<0$, so
similarly as (\ref{eq5-2}) we see that $\|u_{\iy, i}\|_{\la_i}^2\le\|u_i\|_{\la_i}^2$,
that is, $u_{\iy, i}$ has the least energy among all
sign-changing solutions of (\ref{eq5-1}). Hence $u_{\iy, i}$ is a least
energy sign-changing solution of (\ref{eq5-1}),
and $(4)$ holds.

Now we consider $k_1\ge 2$ and $k_2=1$. Then $u_{\iy, 1}$ changes sign as above. Since $u_{\bb, 2}$ is positive, so $u_{\iy, 2}>0$ in $\{u_{\iy, 2}\neq 0\}$.
Define
{\allowdisplaybreaks
\begin{align*}
\mathcal{N}_{k_1,1,\bb}:=\Big\{&\vec{u}=(u_1, u_2)\in H \,\,:\,\, \hbox{$u_1$ changes sign and has at least $k_1$ nodal}\\
&\hbox{domains $\Om_k (1\le k\le k_1)$},\,\, u_2\neq 0,\,\,u_2\ge 0,\\
&E_\bb'(\vec{u})(u_1\chi_{\Om_k}, 0)=0,\,\,\forall\,1\le k\le k_1,\,\, E_\bb'(\vec{u})(0, u_2)=0\Big\}.
\end{align*}
}%
Then the same arguments as (\ref{eq4-4})-(\ref{eq4-5}) yield that
$$ c_{\bb, \dd_\bb}^{k_1, 1}\le \inf_{\vec{u}\in\mathcal{N}_{k_1, 1, \bb}}E_\bb(\vec{u}).$$
If $\{u_{\iy, 1}\neq 0\}$ has at least $k_1+1$ connected components $\Om_k (1\le k\le k_1+1)$, then
$$\left(u_{\iy, 1}\chi_{\cup_{k=1}^{k_1}\Om_k}, \,\,u_{\iy, 2}\right)\in \mathcal{N}_{k_1, 1, \bb}\quad\forall\,\,\bb<0.$$
Similarly as (\ref{eq5-2}) we get a contradiction. Hence, $\{u_{\iy, 1}\neq 0\}$ has at most $k_1$ connected components.

If $(k_1, k_2)=(2, 1)$, then $\{u_{\iy, 1}\neq 0\}$ has exactly two connected components. If $\{u_{\iy, 2}\neq 0\}$ has
at least two connected components $\Om_1$ and $\Om_2$, then $(u_{\iy, 1}, u_{\iy, 2}\chi_{\Om_1})\in \mathcal{N}_{2, 1, \bb}$ for all $\bb<0$, and similarly as (\ref{eq5-2}) we get a contradiction. Hence $\{u_{\iy, 2}\neq 0\}$ is connected.
Finally, similarly as above, we can prove that $u_{\iy, 1}$ is a least energy
sign-changing solution of (\ref{eq5-1}) for $i=1$, and $u_{\iy, 2}$ is a least energy solution of (\ref{eq5-1}) for $i=2$.
This completes the proof.\end{proof}

\section{The entire space case}
\renewcommand{\theequation}{6.\arabic{equation}}

In this final section, we extend some results above to the case where $\Om=\RN$. That is, we consider
the following elliptic system in the entire space
\be\label{eq6-1}
\begin{cases}-\Delta u_1 +\la_1 u_1 =
\mu_1 u_1^3+\beta u_1 u_2^2, \quad x\in \RN,\\
-\Delta u_2 +\la_2 u_2 =\mu_2 u_2^3+\beta u_1^2 u_2,  \quad   x\in\RN,\\
u_1(x),\,u_2(x)\to 0  \,\,\,\hbox{as}\,\,|x|\to +\iy.\end{cases}\ee
By giving some modifications to arguments in Sections 2-4 and introducing some different ideas and techniques, we can prove the following results.

\bt\label{th6-1}Let $N=2, 3$, $\la_1, \la_2$, $\mu_1, \mu_2>0$ and $\bb<0$. Then (\ref{eq6-1}) has infinitely
many radially symmetric sign-changing solutions, including a special $(u_1, u_2)$ such that both $u_1$ and $u_2$
have exactly two nodal domains and $(u_1, u_2)$ has the least energy among all radially symmetric sign-changing solutions.\et

\bt\label{th6-3} Let assumptions in Theorem \ref{th6-1} hold.
Then (\ref{eq6-1}) has infinitely many radially symmetric semi-nodal solutions $\{(u_{n,1}, u_{n, 2})\}_{n\ge 2}$ such that
\begin{itemize}
\item[$(1)$] $u_{n, 1}$ changes sign and $u_{n, 2}$ is positive;
\item[$(2)$] $u_{n, 1}$ has at most $n$ nodal domains. In particular, $u_{2, 1}$ has exactly two nodal domains, and $(u_{2,1}, u_{2,2})$ has the least energy among all nontrivial radially symmetric solutions whose first component changes sign.
\end{itemize}
\et

\br Let assumptions in Theorem \ref{th6-1} hold. Lin and Wei \cite{LW1} proved that (\ref{eq6-1}) has no least energy solutions.
Later, Sirakov \cite{S} proved that (\ref{eq6-1}) has a radially symmetric positive solution which has the least energy among all nontrivial radially symmetric
solutions. Combining these with the introduction in Section 1, our results here are completely new.
\er

Define
$H_r:= H_r^1(\RN)\times H_r^1(\RN)$ as a subspace of $H:=H^1(\RN)\times H^1(\RN)$ with norm $\|(u_1, u_2)\|_{H}^2=\|u_1\|_{\la_1}^2+\|u_2\|_{\la_2}^2$, where
{\allowdisplaybreaks
\begin{align*}
&H_r^1(\RN):=\left\{u\in H^1(\RN) \,\,: \,\,u\,\,\hbox{is radially symmetric}\right\},\\
&\|u\|_{\la_i}^2:=\int_{\RN}(|\nabla u|^2+\la_i u^2)\,dx.
\end{align*}
}%
Since the embedding $H_r^1(\RN)\hookrightarrow L^4(\RN)$ is compact,
by replacing $H_0^1(\Om), H$ with $H_r^1(\RN), H_r$ respectively in all definitions appeared in Sections 2-4 and using the same notations,
it is easy to see that all arguments (with trivial modifications) in Sections 2-4 hold for system (\ref{eq6-1})
except those in Step 2 of proving Theorems \ref{th1} and \ref{th3}.
Hence we only need to reprove Step 2 in the proofs of Theorems \ref{th1} and \ref{th3}.
The following ideas and arguments are quite different from those in Step 2 of proving Theorems \ref{th1} and \ref{th3},
and also can be used in the bounded domain case.

\begin{proof}[Proof of Theorem \ref{th6-1}.] Assume by contradiction that there exists $n_0\in\mathbb{N}$ such that (\ref{eq6-1}) has only $n_0$ radially symmetric sign-changing solutions.
Fix any $k_2\ge 2$, we define
$$l:=\max\left\{c^{k_1, k_2} : 2\le k_1\le n_0+2\right\}+1.$$
For any $k_1\in [2, n_0+2]$ and $0<\dd < 2^{-1/4}$, similarly as (\ref{eq2-17})-(\ref{eq2-17-1}) we define
\be\label{eq6-17}c_{\bb, l,\dd}^{k_1, k_2}:=\inf_{A\in \Gamma_{\bb, l}^{(k_1, k_2)}}\sup_{\vec{u}\in A\setminus \mathcal{P}_\dd}J_\bb(\vec{u}),\ee
where
\be\label{eq6-17-1}\Gamma_{\bb, l}^{(k_1, k_2)}:=\left\{A\in \Gamma^{(k_1, k_2)} \,:\, \sup_{A}J_\bb< l\right\}.\ee
Lemma \ref{lemma7} yields that $\Gamma_{\bb, l}^{(k_1, k_2)}\neq \emptyset$, $c_{\bb, l, \dd}^{k_1, k_2}$ is
well defined and  $c_{\bb, l, \dd}^{k_1, k_2}\le c^{k_1, k_2}$ for each $k_1\in [2, n_0+2]$.
Noting that $\Gamma_{\bb, l}^{(k_1+1, k_2)}\subset\Gamma_{\bb, l}^{(k_1, k_2)}$, we have
\begin{align}\label{eq6-2}
c_{\bb, l,\dd}^{2, k_2}\le c_{\bb, l, \dd}^{3, k_2}\le\cdots\le c_{\bb, l, \dd}^{n_0+1, k_2}\le c_{\bb, l, \dd}^{n_0+2, k_2}.
\end{align}

By repeating the arguments in Section 2 we can prove that, there exists $\dd_{l}\in (0, 2^{-1/4})$ such that
for any $\dd\in (0, \dd_l)$,
\be\label{eq6-3}\eta(t, \vec{u})\in \mathcal{P}_\dd \quad\hbox{whenever}\,\,u\in \mathcal{M}_\bb\cap \mathcal{P}_\dd, \, J_\bb(u)\le l\,\,\hbox{and}\,\,t>0,\ee
and so $c_{\bb, l, \dd}^{k_1, k_2}$ is a radially symmetric sign-changing critical value of $E_\bb$ for each $k_1\in [2, n_0+2]$
(that is,
$E_\bb$ has a radially symmetric sign-changing critical point $\vec{u}$ with $E_\bb (\vec{u})=c_{\bb, l, \dd}^{k_1, k_2}$).
Fix any a $\dd\in (0, \dd_l)$. By (\ref{eq6-2}) and our assumption that
(\ref{eq6-1}) has only $n_0$ radially symmetric sign-changing solutions, there exists some $2\le N_1\le n_0+1$ such that
\be\label{eq2-27}c_{\bb, l,\dd}^{N_1, k_2}= c_{\bb, l,\dd}^{N_1+1, k_2}=:\bar{c}.\ee
Define
\be\label{eq6-10}\mathcal{K}:=\{\vec{u}\in \mathcal{M} \,\,\,:\,\,\,\vec{u}\,\,\hbox{sign-changing},\,\,\, J_\bb(\vec{u})=\bar{c},\,\,\,V(\vec{u})=0 \}.\ee
Then $\mathcal{K}$ is finite. By (\ref{eq2-15}) one
has that $\sg_i(\vec{u})\in \mathcal{K}$ if $\vec{u}\in\mathcal{K}$, that is, $\mathcal{K}\subset\mathcal{F}$.
Hence there exist $k\le n_0$ and $\{\vec{u}_m : 1\le m\le k\}\subset \mathcal{K}$
such that
$$\mathcal{K}=\{\vec{u}_m,\,\sg_1(\vec{u}_m),\, \sg_2(\vec{u}_m),\,-\vec{u}_m\,\,:\,\,\,1\le m\le k\}.$$
Then there exist open neighborhoods $O_{\vec{u}_m}$ of $\vec{u}_m$ in $H_r$, such that
any two of $\overline{O_{\vec{u}_{m_1}}}, \,\sg_1(\overline{O_{\vec{u}_{m_2}}}), \,\sg_2(\overline{O_{\vec{u}_{m_3}}})$ and $-\overline{O_{\vec{u}_{m_4}}}$,
where $1\le m_1, m_2, m_3, m_4\le k$, are disjointed and
$$
\mathcal{K}\subset O:=\bigcup_{m=1}^k
O_{\vec{u}_m}\cup\sg_1(O_{\vec{u}_m})\cup\sg_2(O_{\vec{u}_m})\cup-O_{\vec{u}_m}.
$$
Define a continuous map $\tilde{f}: \overline{O}\to \R\setminus\{0\}$ by
$$\tilde{f}(\vec{u}):=\begin{cases}1,&\hbox{if}\,\, \vec{u}\in \bigcup_{m=1}^k\overline{O_{\vec{u}_m}}\cup\sg_2(\overline{O_{\vec{u}_m}}),\\
-1, &\hbox{if}\,\, \vec{u}\in\bigcup_{m=1}^k \sg_1(\overline{O_{\vec{u}_m}})\cup-\overline{O_{\vec{u}_m}}.\end{cases}$$
Then
$\tilde{f}(\sg_1(\vec{u}))=-\tilde{f}(\vec{u})$ and $\tilde{f}(\sg_2(\vec{u}))=\tilde{f}(\vec{u}).$
By Tietze's extension theorem, there exists $f\in C(H_r, \R)$ such that $f|_{O}\equiv\tilde{f}$.
Define
$$F(\vec{u}):=\frac{f(\vec{u})+f(\sg_2(\vec{u}))-f(\sg_1(\vec{u}))-f(-\vec{u})}{4},$$
then $F|_O\equiv\tilde{f}$,
$F(\sg_1(\vec{u}))=-F(\vec{u})$ and $F(\sg_2(\vec{u}))=F(\vec{u})$.
Define
$$\mathcal{K}_\tau:=\left\{\vec{u}\in \mathcal{M} : \inf_{\vec{v}\in \mathcal{K}}\|\vec{u}-\vec{v}\|_{H}<\tau\right\}.$$
Then we can take small $\tau>0$ such that $\mathcal{K}_{2\tau}\subset O$. Recalling $V(\vec{u})=0$ in $\mathcal{K}$
and $\mathcal{K}$ finite, there exists $\widetilde{C}>0$ such that
\be\label{eq2-28}\|V(\vec{u})\|_{H}\le\widetilde{C},\quad\forall\,\,\vec{u}\in \overline{\mathcal{K}_{2\tau}}.\ee
For any $\vec{u}\in \mathcal{K}_{2\tau}$,
we have $F(\vec{u})=\tilde{f}(\vec{u})\neq 0$. That is
$F(\mathcal{K}_{2\tau})\subset\R\setminus\{0\}$.
By (\ref{eq6-10}) and Lemma \ref{lemma9} there exists small $\e\in (0, 1)$ such that
\be\label{eq2-30}\|V(\vec{u})\|_H^2\ge \e,\quad\forall\,u\in\mathcal{M}_\bb\setminus(\mathcal{K}_{\tau}\cup \mathcal{P}_\dd)
\,\,\,\hbox{satisfying}\,\,\,|J_\bb(\vec{u})-\bar{c}|\le 2\e.\ee
Recall $C_0$ in (\ref{eq2-16}), we let
\be\label{eq2-32}\al:=\frac{1}{2}\min\left\{1, \frac{\tau C_0}{2\widetilde{C}}\right\}.\ee
By (\ref{eq6-17})-(\ref{eq6-17-1}) and (\ref{eq2-27})
we take $A\in \Gamma_{\bb, l}^{(N_1+1, k_2)}$ such that
\be\label{eq2-33}\sup_{A\setminus \mathcal{P}_{\dd}}J_\bb<c_{\bb, l, \dd}^{N_1+1, k_2}+\al\e=\bar{c}+\al\e.\ee
Let $B:=A\setminus \mathcal{K}_{2\tau}$, then it is easy to check that $B\subset\mathcal{F}$. We claim that $\vec{\ga}(B)\ge (N_1, k_2)$. If not, there exists
$\tilde{g}\in F_{(N_1, k_2)}(B)$ such that $\tilde{g}(\vec{u})\neq 0$ for any $\vec{u}\in B$.
By Tietze's extension theorem, there exists $\bar{g}=(\bar{g}_1, \bar{g}_2)\in C(H_r, \R^{N_1-1}\times\R^{k_2-1})$ such that $\bar{g}|_{B}\equiv\tilde{g}$.
Define $g=(g_1, g_2)\in C(H_r, \R^{N_1-1}\times\R^{k_2-1})$ by
{\allowdisplaybreaks
\begin{align*}g_1(\vec{u}):=
\frac{\bar{g}_1(\vec{u})+\bar{g}_1(\sg_2(\vec{u}))-\bar{g}_1(\sg_1(\vec{u}))-\bar{g}_1(-\vec{u})}{4},\\
g_2(\vec{u}):=\frac{\bar{g}_2(\vec{u})+\bar{g}_2(\sg_1(\vec{u}))-\bar{g}_2(\sg_2(\vec{u}))-\bar{g}_2(-\vec{u})}{4},
\end{align*}
}%
then $g|_{B}\equiv\tilde{g}$,
$g_i(\sg_i(\vec{u}))=-g_i(\vec{u})$ and $g_i(\sg_j(\vec{u}))=g_i(\vec{u})$ for $j\neq i$.
Finally we define $G=(G_1, G_2)\in C(A,\, \R^{N_1+1-1}\times\R^{k_2-1})$ by
$$G_1(\vec{u}):=(F(\vec{u}),\, g_1(\vec{u}))\in \R^{N_1+1-1},\quad G_2(\vec{u}):=g_2(\vec{u})\in\R^{k_2-1}.$$
By our constructions of $F$ and $g$, we have $G\in F_{(N_1+1, k_2)}(A)$. Since $\vec{\ga}(A)\ge (N_1+1, k_2)$, so $G(\vec{u})=0$ for some $\vec{u}\in A$.
If $\vec{u}\in \mathcal{K}_{2\tau}$, then $F(\vec{u})\neq 0$, a contradiction. So $\vec{u}\in A\setminus\mathcal{K}_{2\tau}=B$,
and then $g(\vec{u})=\tilde{g}(\vec{u})\neq 0$, also a contradiction. Hence $\vec{\ga}(B)\ge (N_1, k_2)$.
Note that $\sup_{B}J_\bb\le \sup_{A}J_\bb<l$,
we see that $B\subset\mathcal{M}_\bb$ and $B\in \Gamma_{\bb, l}^{(N_1, k_2)}$.
Then we can consider $D:=\eta(\tau/(2\widetilde{C}), B)$, where $\eta$ is in Lemma \ref{lemma11} and $\widetilde{C}$ is in (\ref{eq2-28}). By Lemma \ref{lemma5}-$(ii)$ and Lemma \ref{lemma11} we have $D\in \Gamma^{(N_1, k_2)}$ and
$\sup_{D} J_\bb\le\sup_B J_\bb<l$, that is $D\in \Gamma_{\bb, l}^{(N_1, k_2)}$.
Then we see from (\ref{eq6-17})-(\ref{eq6-17-1}) and (\ref{eq2-27}) that
$$\sup_{D\setminus \mathcal{P}_{\dd}}J_\bb\ge c_{\bb, l, \dd}^{N_1, k_2}=\bar{c}.$$
By Lemma \ref{lemma6} we can take $\vec{u}\in B$ such that $\eta(\tau/(2\widetilde{C}), \vec{u})\in D\setminus \mathcal{P}_{\dd}$ and
$$\bar{c}-\al\e\le \sup_{D\setminus \mathcal{P}_\dd}J_\bb-\al\e< J_\bb (\eta(\tau/(2\widetilde{C}), \vec{u})).$$
Since $J_\bb (\eta(t, \vec{u}))\le J_\bb(\vec{u})<l$ for any $t\ge 0$, (\ref{eq6-3}) yields $\eta(t, \vec{u})\not\in\mathcal{P}_\dd$ for any $t\in [0, \tau/(2\widetilde{C})]$. In particular, $\vec{u}\not\in\mathcal{P}_{\dd}$ and so (\ref{eq2-33}) yields $J_\bb(\vec{u})< \bar{c}+\al\e$. Then
for any $t\in [0, \tau/(2\widetilde{C})]$, we have
$$\bar{c}-\al\e< J_\bb (\eta(\tau/(2\widetilde{C}), \vec{u}))\le J_\bb (\eta(t, \vec{u}))\le J_\bb (\vec{u})< \bar{c}+\al\e.$$
Recall that $\vec{u}\in B=A\setminus\mathcal{K}_{2\tau}$. If there exists $T\in (0, \tau/(2\widetilde{C}))$ such
that $\eta(T, \vec{u})\in \mathcal{K}_\tau$, then
there exist $0\le t_1<t_2\le T$ such that $\eta(t_1, \vec{u})\in\partial\mathcal{K}_{2\tau}$, $\eta(t_2, \vec{u})\in\partial\mathcal{K}_{\tau}$ and
$\eta(t, \vec{u})\in\mathcal{K}_{2\tau}\setminus\mathcal{K}_\tau$ for any $t\in (t_1, t_2)$. So we see from (\ref{eq2-28}) that
$$\tau\le\|\eta(t_1, \vec{u})-\eta(t_2, \vec{u})\|_H=\left\|\int_{t_1}^{t_2}V(\eta(t, \vec{u}))\,dt\right\|_H\le 2\widetilde{C}(t_2-t_1),$$
that is, $\tau/(2\widetilde{C})\le t_2-t_1\le T$, a contradiction. Hence $\eta(t, \vec{u})\not\in \mathcal{K}_\tau$ for any $t\in (0, \tau/(2\widetilde{C}))$.
Then as Step 1 in the proof of Theorem \ref{th1}, we deduce from (\ref{eq2-30}) and (\ref{eq2-32}) that
\begin{align*}
\bar{c}-\al\e< J_\bb (\eta(\tau/(2\widetilde{C}), \vec{u}))\le J_\bb(\vec{u})-\int_0^{\tau/(2\widetilde{C})}C_0\e\,dt<\bar{c}+\al\e-2\al\e=\bar{c}-\al\e,
\end{align*}
a contradiction.
Hence (\ref{eq6-1}) has infinite many radially symmetric sign-changing solutions.
This completes the proof.\end{proof}

\begin{proof}[Proof of Theorem \ref{th6-3}.] It suffices to prove that (\ref{eq6-1}) has infinitely
many semi-nodal solutions. This argument is similar as above, we omit the details.\end{proof}

\noindent{\it Acknowledgements.} The authors wish to thank the anonymous referee very much for careful reading and valuable comments.

\end{document}